\documentclass[12pt,a4paper,twoside]{article}

\usepackage[centertags]{amsmath}
\usepackage{amsfonts}
\usepackage{amssymb}
\usepackage{ulem}
\usepackage{epsfig}
\usepackage{subfigure}
\usepackage{color}
\usepackage{mathrsfs}
\usepackage[boldsans]{ccfonts}
\usepackage{array}
\usepackage{calc}
\usepackage{graphicx}
\usepackage{enumitem}
\usepackage{amsthm}
\usepackage{hyperref}

\newtheoremstyle{break}
  {\topsep}{\topsep}%
  {\itshape}{}%
  {\bfseries}{}%
  {\newline}{}%
\theoremstyle{break} \newtheorem{theorem}{Theorem}[section]
\theoremstyle{break} \newtheorem{satz}{Theorem}
\theoremstyle{break} \newtheorem{lemma}[theorem]{Lemma}
\theoremstyle{break} 
\theoremstyle{break}
\newtheorem{remark}[theorem]{Remark}
\numberwithin{equation}{section}
\definecolor{orange}{rgb}{1,0.5,0}

\newcommand{\hide}[1]{}
\newcommand{\N}{{\mathbb{N}}}
\newcommand{\R}{{\mathbb{R}}}
\newcommand{\D}{{\mathbb{D}}}
\newcommand{\Ha}{{\mathbb{H}}}
\newcommand{\C}{{\mathbb{C}}}

\renewcommand{\H}{{\mathbb{H}}}
\newcommand{\hcap}{\operatorname{hcap}}
\newcommand\with{\ \vrule\ }  
\newcommand{\LandauO}{\mathcal{O}} 
\newcommand{\Landauo}{{\scriptstyle\mathcal{O}}}

\newcommand{\diam}{\operatorname{diam}}
\def\supp{\mathop{{\rm supp}}}

\def\Re{\mathop{{\rm Re}}}
\def\Im{\mathop{{\rm Im}}}

\newcommand{\eps}{\varepsilon}

\begin{document}

 \setcounter{section}{0}

   \author{Oliver Roth and Sebastian Schlei\ss inger}
   \title{\vspace*{-0.8cm}The Schramm--Loewner equation for multiple~slits}
   \date{\today}
   \maketitle

\begin{abstract}
     We prove that any disjoint union of finitely many simple curves in the upper half--plane
     can be generated in a unique way by the chordal multiple--slit Loewner equation with constant weights.
\end{abstract}

\section{Introduction and results}

Recent progress in the theory of Loewner equations
(\cite{Loewner:1923,Schramm:2000,Lawler:2005,BCD:2012})
suggests that
one of the most useful descriptions of a simple plane curve is by encoding it into a
growth process modeled by the Schramm--Loewner equation. In this paper we show that
any disjoint union of \textit{finitely many} simple curves can be encoded in a
\textit{unique} way into a growth process described by a multi--slit version of
the Schramm--Loewner equation. In order to state our result we need to introduce some notation.

\medskip

Let $\Ha:=\{z\in\C \with \Im(z)>0\}$ be the upper half--plane. A \textit{slit}
is the trace $\Gamma=\gamma(0,1]$ of a simple curve  $\gamma
:[0,1] \to \overline{\H}$  with $\gamma(0) \in \R$ and $\Gamma \subseteq \H$. Since $\H \backslash \Gamma$ is a simply connected
domain, (a
version of) Riemann's mapping theorem guarantees that there is a unique
conformal map $g_{\Gamma}$ from $\H \backslash \Gamma$ onto $\H$ with
\textit{hydrodynamic normalization}
$$ g_{\Gamma}(z)=z+\frac{b}{z}+\LandauO(|z|^{-2}) \quad \text{ as } \quad z \to
\infty \, $$
for some $b>0$. We call $\hcap(\Gamma):=b$ the \textit{half--plane capacity}
of the slit $\Gamma$. The following well--known result provides a description
of the slit $\Gamma$ with the help of the chordal
Loewner equation (Schramm--Loewner equation).

\begin{satz}[The one--slit chordal Loewner equation] \label{satz:1} \label{slitex}
Let $\Gamma$ be a slit with $\hcap(\Gamma)=2T$. Then there exists a unique
continuous driving function $U : [0,T] \to \R$ such that the solution $g_t$ to
the chordal Loewner equation
\begin{equation}\label{ivp}
   \dot{g}_t(z)=\frac{2}{g_t(z)-U(t)}, \quad g_0(z)=z,
\end{equation}
has the property that $g_T=g_{\Gamma}$.
\end{satz}

Note that in Theorem \ref{satz:1} the slit $\Gamma$ ``starts'' at the point $U(0)$.

\medskip

To the best of our knowledge the first proof of Theorem \ref{satz:1} is due
to Kufarev, Sobolev, Spory{\v{s}}eva in \cite{MR0257336}. The basic recent
reference for Theorem \ref{satz:1} is the book of Lawler \cite{Lawler:2005}. 
We also refer to the survey paper \cite{GM} for a complete and rigorous
proof of Theorem \ref{satz:1} using classical complex analysis.

\medskip

Now, let $\Gamma=(\Gamma_1, \ldots, \Gamma_n)$ be a \textit{multi--slit},
that is, the union of $n$  slits $\Gamma_1, \ldots,
\Gamma_n$ with disjoint closures. As before, there is a unique conformal map $g_{\Gamma}$ from $\H
\backslash \Gamma$ onto $\H$ with expansion
$$g_{\Gamma}(z)=z+b/z+\LandauO(|z|^{-2}) \quad \text{ as } \quad z \to
\infty\, , $$ and we call $\hcap(\Gamma):=b>0$
the half--plane capacity of $\Gamma$. The main result of the present paper is
the following extension of Theorem \ref{satz:1}.

\begin{theorem}\label{Charlie2}
 Let $\Gamma=(\Gamma_1,\ldots, \Gamma_n)$ be a multi--slit with $\hcap(\Gamma)=2T$.
Then there exist unique weights $\lambda_1,...,\lambda_n\in [0,1]$ with
$\sum_{k=1}^n\lambda_k=1$ and unique continuous driving functions
$U_1,...,U_n:[0,T]\to\R,$ such that the solution $g_t$ of the chordal Loewner equation \begin{equation}\label{more}\dot{g}_t(z)=\sum_{k=1}^n\frac{2\lambda_k}{g_t(z)-U_{k}(t)},\qquad g_0(z)=z,\end{equation} 
satisfies $g_T=g_{\Gamma}$. 
\end{theorem}

Some remarks are in order.

\begin{remark}[The multi--slit chordal Loewner equation]
{\rm It is well--known and
easy to prove on the basis of Theorem \ref{satz:1} that under the conditions
of Theorem \ref{Charlie2}  there are
\begin{itemize}
\item[(a)] $n$
continuous weight
functions $\lambda_1, \ldots, \lambda_n :[0,T] \to \R$ with $\lambda_j(t) \ge
0$ and $\lambda_1(t)+\ldots+ \lambda_n(t)=1$ for every $t \in [0,T]$, and
\item[(b)] $n$ continuous driving functions $U_1, \ldots, U_n : [0,T]\to \R$,
\end{itemize}
such that the solution $g_t$ to the  \textit{multi--slit
chordal Loewner equation}
\begin{eqnarray}\label{multifinger}
\dot{g_t}(z)=\sum_{j=1}^n\frac{2\lambda_j(t)}{g_t(z)-U_j(t)},\quad g_0(z)=z\, ,
\end{eqnarray}
satisfies $g_T=g_{\Gamma}$, see Remark \ref{thm:twoslit} below.
  However, the $2n$ functions $\lambda_j(t)$ and $U_j(t)$
are \textit{not} uniquely determined by the multi--slit $\Gamma$ if $n>1$,
simply because in this case there are obviously many Loewner chains
$\tilde{g}_t$ (in the sense of \cite{Lawler:2005}) such that $\tilde{g}_T=g_{\Gamma}$.
Informally, each weight function $\lambda_j(t)$ corresponds to the speed
of growth of the slit $\gamma_j$ (the one that starts at the point $U_j(0)$).
Theorem \ref{Charlie2} shows that one can actually choose \textit{constant}
weight functions $\lambda_j$, which are moreover \textit{uniquely determined}.
 In addition,  then also the driving functions $U_1(t), \ldots, U_n(t)$
are uniquely determined by the multi--slit $\Gamma$. Hence Theorem
\ref{Charlie2} 
provides a
canonical way of describing a multi--slit by a growth process modeled via a
Loewner--type equation. We therefore call the differential equation (\ref{more}), i.e., the
multi--slit chordal Loewner equation with constant weights, the
\textit{Schramm--Loewner equation for the multi--slit~$\Gamma$}. }
\end{remark}

\begin{remark}[The multi--slit Loewner equation in Mathematical Physics]
{\rm We note in passing that the multiple--slit equation (\ref{multifinger}) has
recently been used in mathematical physics
 for the study of certain two--dimensional growth phenomena. 
For instance, in \cite{MR1945279}  the authors analyze ``Laplacian path models'', i.e.~Laplacian growth models for multi--slits. By mapping the upper half--plane conformally onto a half-strip one obtains a Loewner equation for the growth of slits in a half--strip, which can be used to describe Laplacian growth in the ``channel geometry'', see \cite{MR2495460} and \cite{PhysRevE.84.051602}. 
Furthermore, equation (\ref{multifinger}) can be used to model so--called
multiple Schramm--Loewner evolutions, see \cite{MR2310306} and
\cite{1038.82074}, \cite{MR2187598}, \cite{MR2358649}, \cite{Graham2007}.} 
\end{remark}

\begin{remark}[The multi--finger radial Loewner equation; Prokhorov's theorem]
{\rm  For the radial Loewner equation on the unit disk
  $\D:=\{z \in \C \with |z|<1\}$, the multi--slit situation
has already been  studied long time ago by Peschl \cite{Peschl} in 1936. He proved
that for every union $\Gamma$ of $n$ Jordan arcs $\Gamma_1, \ldots, \Gamma_n$ in $\overline{\D} \backslash \{0\}$ such that 
$\D \backslash \Gamma$ is simply connected, there are continuous weight functions
$\lambda_1, \ldots, \lambda_n : [0,T] \to \R$ with $\lambda_j(t) \ge 0$ and
$\lambda_1(t)+\ldots+\lambda_n(t)=1$ for every $t \in [0,T]$, and continuous
driving functions $\kappa_j : [0,T] \to \partial \D$ such that the solution
$w_t$ to
the radial Loewner equation
\begin{equation} \label{eq:rad}
 \dot{w}_t(z)=-w_t(z)\sum \limits_{j=1}^n \lambda_j(t)
\frac{\kappa_j(t)-w_t(z)}{\kappa_j(t)-w_t(z)} \, , \qquad w_0(z)=z \, ,\end{equation}
has the property that $w_T$ maps $\D$ conformally onto $\D \backslash \Gamma$.
As in the chordal case, this representation of the multi--slit $\Gamma$ is not
unique.
However, if 
the Jordan arcs $\Gamma_1, \ldots, \Gamma_n$ are \textit{piecewise analytic}, it is has been proved by D.~Prokhorov \cite[Theorem 1 \& 2]{Prokhorov:1993} that
 one can choose constant weight functions and that then these weights as
well as the continuous driving functions are uniquely determined. 
Prokhorov's result 
forms the basis for his original and penetrating control--theoretic study of extremal problems for
univalent functions, see his monograph \cite{Prokhorov:1993}. Clearly, Prokhorov's result is the analogue of Theorem
\ref{Charlie2} for the radial Loewner equation (\ref{eq:rad}), but only under the
very restrictive additional assumption that the multi--slit is piecewise
analytic. An extension of Prokhorov's theorem for not necessarily
piecewise analytic slits, i.e., the full analogue of Theorem \ref{Charlie2}
for the radial case will be discussed in the forthcoming paper \cite{BS}.}
\end{remark}

\hide{\begin{remark}
{\rm Our proof of Theoerem \ref{Charlie2} only works under the asumption that the
slits $\Gamma_1\,\ldots, \Gamma_n$ are disjoint. We strongly suspect, however, that it
suffices to assume $\Gamma_1, \ldots, \Gamma_n$ are slits such that
$\H \backslash (\Gamma_1 \cup \ldots \cup \Gamma_n)$ is simply connected.}
\end{remark}}

\begin{remark}[Schramm--Loewner constants]
{\rm We call the constant weights $\lambda_1,\ldots, \lambda_n$ in Theorem
  \ref{Charlie2} the \textit{Schramm--Loewner constants of the multi--slit  $\Gamma$}.
 Is there a interpretation
for the Schramm--Loewner constants in terms of \textit{geometric} or
\textit{potential theoretic} properties of $\Gamma$\,?
Since our proof of Theorem \ref{Charlie2} is non--constructive, it would be
interesting to find a method for computing the Schramm--Loewner constants for a given
multi--slit $\Gamma$.
}
\end{remark}

We will now outline  the main idea of the proof of Theorem \ref{Charlie2}
(Existence) for the case of a \textit{two--slit} $(\Gamma_1 ,\Gamma_2)$.
 Roughly speaking, we use a ``Bang--Bang Method'' based on the
one--slit Loewner equation (\ref{ivp}). Let $\Gamma_1$ and $\Gamma_2$ be two slits with disjoint closures. We can 
assume  $\hcap(\Gamma_1 \cup \Gamma_2)=2$. By extending $\Gamma_1$ and
$\Gamma_2$,
 we can find two slits 
 $\Theta_1\supseteq \Gamma_1$ and $\Theta_2\supseteq \Gamma_2$ with disjoint closures such that 
 $\hcap(\Theta_1)=\hcap(\Theta_2)=2$.

\medskip

\textit{Step 1:} Let $\alpha : [0,1] \to \{0,1\}$ be a step function. We construct two
continuous driving functions $U_{1,\alpha}, U_{2,\alpha} : [0,1] \to \R$ such that
the solution  to the Loewner equation
\begin{equation}\label{e}
  \dot{g}_{t,\alpha}(z)=\frac{2\alpha(t)}{g_{t,\alpha}(z)-U_{1,\alpha}(t)}+\frac{2(1-\alpha(t))}{g_{t,\alpha}(z)-U_{2,\alpha}(t)}\,
  ,
  \quad g_{0,\alpha}(z)=z \, , \end{equation}
at time $t=1$ satisfies $g_{1,\alpha}=g_{A_{\alpha}}$, where
 the two--slit $A_{\alpha}$ is a subset of $\Theta_1 \cup \Theta_2$.
Informally, the two--slit $A_{\alpha}$ is generated by letting $\Theta_1$ grow
whenever $\alpha=1$, and by letting $\Theta_2$ grow  whenever $\alpha=0$.
Note that (\ref{e}) has the form of the one--slit Loewner equation (\ref{ivp})
but with a discontinuous (``bang--bang'') driving function.

\medskip
\textit{Step 2:} We show that  the set 
of all driving functions from Step 1
is  a precompact subset of  the Banach space $C[0,1]$ of
continuous functions on $[0,1]$ equipped with the sup--norm
$||\cdot||_{\infty}$. The proof of this key observation requires a fair amount
of technical work, which will be carried out in Section 2 and Section 3.
\medskip

\textit{Step 3:} We construct a sequence of step functions $\alpha_n: [0,1]
\to \{0,1\}$ such that:
\begin{itemize}
\item[(i)] For every $n \in \N$ the two--slit $A_{\alpha_n}\subseteq \Theta_1 \cup \Theta_2$ generated by the step
  function $\alpha_n$ via Step 1 is exactly the two--slit $\Gamma_1 \cup \Gamma_2$.
\item[(ii)] The sequence $(\alpha_n)$ converges weakly in the Banach space $L^1[0,1]$ to a
  constant $\lambda \in [0,1]$.
\end{itemize}
Each  step function $\alpha_n$ is constructed as follows. We divide
$[0,1]$ into $2^n$ disjoint intervals of equal length and let $\mu \in
[0,1]$. On each of these
intervals we let $\Theta_1$ grow on the first subinterval of length $\mu/2^n$
and we let $\Theta_2$ grow on the remaining subinterval of length $(1-\mu)/2^n$.
A continuity argument shows that there is a number $\mu_n \in [0,1]$ such that this process
generates exactly the two--slit $(\Gamma_1,\Gamma_2)$. Passing to a
subsequence if necessary, we may assume that $(\mu_n)$ is convergent with limit $\lambda$. The corresponding step functions $\alpha_n$ then do
have the required properties (i) and (ii).

\medskip

\textit{Step 4:} Using the step functions $\alpha_n$ of
Step 3, we construct the corresponding driving functions $U_{1,\alpha_n}$ and
$U_{2,\alpha_n}$ by Step 1. With the help of Step 2, 
 we get subsequential limit functions $U_1, U_2 \in C[0,1]$ and finally show
that the solution $g_t$ to 
$$ \dot{g}_{t}=\frac{2\lambda}{g_{t}-U_{1}(t)}+\frac{2(1-\lambda)}{g_{t}-U_{2}(t)},
  \quad g_{0}(z)=z \, $$
has the property that $g_1=g_{\Gamma_1 \cup \Gamma_2}$.

\medskip

This paper is organized as follows.
In Sections 2 and 3 we provide a number of technical, but crucial  auxiliary
results, which will be used in Section 4  for the proof of the existence statement of Theorem
\ref{Charlie2}. In Section 5 we establish a dynamic
interpretation of the weights $\lambda_1,\ldots, \lambda_n$, which will be employed
for the proof of the uniqueness statement  of Theorem \ref{Charlie2} in
Section 6. We shall give the details only in the case $n=2$, i.e., for two slits.
The general case of $n \ge 2$ slits can be proved in exactly the
same way by induction. 

\section{The two--slit chordal Loewner equation} \label{sec:two--slit}

We first recall that a bounded subset $A\subset\Ha$  is called a \textit{hull} if 
 $A=\Ha\cap \overline{A}$ and $\Ha\setminus A$ is simply connected, so every
 slit and every multi--slit is a hull.
 For a hull $A$ we denote 
by $g_A$  the unique conformal
 mapping from $\Ha\setminus A$ onto $\Ha$ such that $$g_A(z)=z+\frac{b}{z}+\LandauO\left(|z|^{-2}\right) \quad
 \text{for} \quad |z|\to\infty \, ,$$ where $\hcap(A):=b \ge 0$ is the 
half--plane capacity of $A$.

\medskip

Now, let $\Gamma_1$ and $\Gamma_2$ be slits such that $\Gamma_1 \cup \Gamma_2$ is a
hull.
We call a pair $(\gamma_1,\gamma_2)$ of continuous functions 
$\gamma_j : [0,1] \to \overline{\H}$  with $\gamma_j(0,1]=\Gamma_j$, $j=1,2$,  a
\textit{Loewner parametrization for the hull $\Gamma_1 \cup\Gamma_2$},
if the following two conditions hold:
\begin{itemize}
\item[(i)]
\, Both  functions, $t \mapsto \hcap (\gamma_1(0,t])$ and $t \mapsto
\hcap(\gamma_2(0,t])$, are nondecreasing;
\item[(ii)] 
\, $\hcap (\gamma_1(0,t] \cup \gamma_2(0,t])=2t $ for every $t \in [0,1]$.
\end{itemize}
Informally, $\gamma_1(0,t] \cup \gamma_2(0,t]$, $t \in [0,1]$, is a continuously increasing
family of subhulls  of $\Gamma_1 \cup \Gamma_2$ such that for every $t \in
[0,1]$ at least one of the two slits 
 is growing.
The functions
\begin{eqnarray*}
 \lambda_1(t)&:=&\frac{1}{2} \frac{d}{ds} \bigg|_{s=0} \hcap \left(
\gamma_1(0,t+s] \cup \gamma_2(0,t] \right) \, ,\\
 \lambda_2(t)&:=&\frac{1}{2} \frac{d}{ds} \bigg|_{s=0} \hcap \left(
\gamma_1(0,t] \cup \gamma_2(0,t+s] \right) \, . 
\end{eqnarray*}
are called the \textit{weight functions of the Loewner parametrization}
$(\gamma_1,\gamma_2)$. Note that $\lambda_1(t), \lambda_2(t)$ are well defined for a.e.~$t \in [0,1]$ as
derivatives of nondecreasing functions and they belong to the space $L^1[0,1]$ of
$L^1$--functions on the interval $[0,1]$.
Moreover, $0 \le \lambda_j(t) \le 1$ and
$\lambda_1(t)+\lambda_2(t)=1$ for a.e.~$t \in [0,1]$ by (ii).
Informally, $\lambda_1(t)$ and $\lambda_2(t)$ measure the speed of growth
of $\gamma_1(t)$ and $\gamma_2(t)$ w.r.t.~half--plane capacity.
If we let $g_t:=g_{\gamma_1(0,t] \cup \gamma_2(0,t]}$, then the functions
 $$ U_1(t):=g_{t}(\gamma_1(t)) \, , \qquad
U_2(t):=g_{t}(\gamma_2(t)) \, ,  $$
are called the \textit{driving functions of the Loewner parametrization}
$(\gamma_1,\gamma_2)$. As in the one--slit case, the driving functions
are continuous (see also Theorem \ref{thm:aux}).

\medskip

If $(\gamma_1, \gamma_2)$ is a Loewner parametrization, then the evolution of the family of subhulls
$\gamma_1(0,t] \cup \gamma_2(0,t]$ can be described by the
two--slit chordal Loewner equation as follows. 
 
\begin{remark}[The two--slit chordal Loewner equation] \label{thm:twoslit}
Let $\Gamma_1, \Gamma_2$ be slits such that $\Gamma_1 \cup \Gamma_2$ is a hull and
let $(\gamma_1, \gamma_2)$ be a Loewner parametrization of $\Gamma_1 \cup
\Gamma_2$ with weight functions $\lambda_1,\lambda_2$ and driving functions $U_1,U_2$.
 Then  the conformal map $g_t:=g_{\gamma_1(0,t] \cup \gamma_2(0,t]}$ is the solution of the Loewner equation
\begin{equation}  \label{eq:twoslit}
\begin{array}{rcl}
\dot{g}_t(z) &=& \displaystyle \frac{2 \lambda_1(t)}{g_t(z)-U_1(t)}+\frac{2
  \lambda_2(t)}{g_t(z)-U_2(t)} \quad \text{ for a.e. } t \in [0,1] \, , \\[4mm]
g_0(z) &=& z \, .
\end{array}
\end{equation}
\end{remark}

A  proof of Remark \ref{thm:twoslit} can be given along the lines
of the proof of Theorem \ref{slitex} in \cite{GM}. We do not give the details
here mainly because  we  need the statement of Remark \ref{thm:twoslit} only  in a very
special case, which can be deduced fairly quickly from the one--slit Loewner
equation (see Lemma \ref{lem:stepfunctions} below). In particular, 
 the proof of Theorem \ref{Charlie2} does \textit{not} depend on
Remark \ref{thm:twoslit}, but \textit{only} on Theorem \ref{satz:1}.

\medskip

Note that, in view of Remark \ref{thm:twoslit}, for proving 
the existence part of Theorem \ref{Charlie2}, we essentially have to show that every two--slit
$(\Gamma_1,\Gamma_2)$ with $\hcap(\Gamma_1 \cup \Gamma_2)=2$
has a Loewner parametrization $(\gamma_1,\gamma_2)$ with
\textit{constant} weight functions.
To this end, we arbitrarily  choose two slits $\Theta_1 \supseteq \Gamma_1$ and
$\Theta_2 \supseteq \Gamma_2$
with disjoint closures such that $\hcap(\Theta_1)=\hcap(\Theta_2)=2$,
and consider all possible Loewner
 paramaterizations of \textit{subhulls} of
$\Theta_1 \cup \Theta_2$. The key result is 
then the following theorem.

\begin{theorem} \label{thm:aux}
Let $\Theta_1, \Theta_2$ be slits with disjoint closures and 
$\hcap(\Theta_1)=\hcap(\Theta_2)=2$. Then the set
of driving functions for all Loewner parametrizations of  subhulls of
$\Theta_1 \cup\Theta_2$ is  a compact subset of the Banach space $C[0,1]$.
\end{theorem}

The proof of Theorem \ref{thm:aux} is divided into two parts.
In this section, we show that the driving functions in Theorem \ref{thm:aux}
form a closed subset of $C[0,1]$, and we defer the more difficult proof of precompactness to
Section 3.  

\medskip

We shall need the following partial converse of Remark \ref{thm:twoslit},
which is actually a special  case of Theorem 4.6 in \cite{Lawler:2005}.

\begin{lemma} \label{lem:q}
 Let $\lambda_1,\lambda_2 \in L^1[0,1]$ with
 $0 \le \lambda_j(t) \le 1$ and $\lambda_1(t)+\lambda_2(t)=1$ for
a.e.~$t \in [0,1]$, and let $U_1,U_2 \in C[0,1]$. 
For every $z \in \H$ let $T_z$ be the supremum of all $t \in [0,1]$ such that the
solution $g_t(z)$ of the initial value problem (\ref{eq:twoslit}) is well
defined up to time $t$ with $g_t(z) \in \H$. Let $H_t:=\{z \in \H \, : \, T_z>t\}$.
Then $g_t$ is the unique conformal map from $H_t$ onto $\H$ such that
$g_t(z)=z+2t/z+O(1/|z|^2)$ as $z \to \infty$. 
\end{lemma}

We call $g_t$ the Loewner chain
associated to the weight functions $\lambda_1,\lambda_2$ and the driving
functions $U_1,U_2$. The following result  shows that the Loewner chain $g_t$ depends continuously
on its weight functions and its driving functions, provided we choose the
appropriate topologies. Recall that a sequence of Loewner chains $g^{(n)}_t$
 is said to converge to the Loewner chain $g_t$ with domain $H_t$ in the Carath\'eodory sense, if for every
 $\eps>0$, $g^{(n)}_t$ converges to $g_t$ uniformly on $[0,1] \times \{z \in
 \H \, : \, \text{dist}(z,K_1) \ge \eps \}$, where $K_1$ is the closure of $\H
 \backslash H_1$, see \cite[\S 4.7]{Lawler:2005}.

\begin{theorem}[Continuous dependence of Loewner chains] \label{thm:control}
For $j\in\{1,2\}$ let $\lambda^{(n)}_{j}, \lambda_{j} \in L^1[0,1]$ be weight
functions and let $U^{(n)}_{j},
U_{j} \in C[0,1]$  be driving functions with associated Loewner
chains  $g^{(n)}_{t}$, $g_t$.
If $U^{(n)}_j$ converges to $U_j$ uniformly on $[0,1]$ and if
$\lambda^{(n)}_j$ converges weakly in $L^1[0,1]$ to $\lambda_j$  for $j=1,2$, then
$g^{(n)}_t$ converges in the Carath\'eodory sense to the chain $g_t$.
\end{theorem}

\begin{remark}
Theorem \ref{thm:control} generalizes Proposition 4.47 in \cite{Lawler:2005},
which deals with the one--slit version of Loewner's equation.
The idea of the statement and the proof of Theorem \ref{thm:control} comes from a
 standard result in linear control theory (see \cite[p.~117]{Jur97})  by thinking of the weight functions
 as ``control functions''.
\end{remark}

\begin{proof}[Proof of Theorem \ref{thm:control}]
For every $\delta>0$, let
$$ V_{\delta}:=\{z \in \H \, : \, |g_t(z)-U_j(t)|> \delta \text{ for } 0 \le t
\le 1, \, j=1,2\} \, .$$
Then $V_{\delta}$ is an open subset of $H_1$. As in
\cite[p.~115]{Lawler:2005}, it suffices to show that $g^{(n)}_t$ converges to
$g_t$ uniformly on $[0,1] \times V_{\delta}$.

\smallskip

We first need to establish a number of technical, but crucial estimates.

\smallskip
(i) Let
$$ \alpha_n(t,z):=\int \limits_{0}^t \left[ \frac{2
    (\lambda_1(s)-\lambda_1^{(n)}(s))}{g_s(z)-U_1(s)}+\frac{2
    (\lambda_2(s)-\lambda_2^{(n)}(s))}{g_s(z)-U_2(s)} \right] \, ds \, .$$
Since $\lambda^{(n)}_j$ converges weakly to $\lambda_j$, we have $\alpha_n(t,z) \to 0$ as
$n \to \infty$ pointwise on $[0,1] \times \H$. In fact, this
convergence is uniform on $[0,1] \times V_{\delta}$, since the sequence
$(\alpha_n)$ is equicontinuous there. This follows from 
\begin{eqnarray*}
|\alpha_n(t,z) \hspace*{-0.3cm}&-&\hspace*{-0.3cm} \alpha_n(t',z')|  \le 
|\alpha_n(t,z)-\alpha_n(t,z')|+|\alpha_n(t,z')-\alpha_n(t',z')| \\
&\le &  \int \limits_0^t \left[\frac{2 |g_s(z)-g_s(z')|}{|g_s(z)-U_1(s)| \,
  |g_s(z')-U_1(s)|} +\frac{2 |g_s(z)-g_s(z')|}{|g_s(z)-U_2(s)| \,
  |g_s(z')-U_2(s)|} \right]\, ds
\\ 
& & \qquad 
+ \left| \int \limits_{t'}^{t} \left[ \frac{2
      (\lambda_1(s)-\lambda_1^{(n)}(s))}{g_s(z')-U_1(s)}+\frac{2
      (\lambda_s(s)-\lambda_2^{(n)}(s))}{g_s(z')-U_2(s)} \right] \, ds
\right|\\
& \le & \frac{4}{\delta^2} \int \limits_{0}^t | g_s(z)-g_s(z')| \,
ds+\frac{4}{\delta} \, |t-t'| 
\end{eqnarray*}
for all $t,t' \in [0,1]$ and $z,z' \in V_{\delta}$.

\smallskip

(ii) Let $$\beta_n:=\max\left\{|U_j(s)-U^{(n)}_j(s)| \, : \, s \in [0,1], \,
j=1,2\right\} \, , $$ so $\beta_n \to 0$ as $n \to \infty$ by assumption. Let $\eps>0$
with $\eps<\delta/4$. Then there is a positive integer $N$ such
that $\beta_n<\delta/4$ for all $n \ge N$. Since $\alpha_n \to 0$ uniformly on
$[0,1] \times V_{\delta}$ by (i), we may assume by enlarging $N$ if necessary that
\begin{equation} \label{eq:w1}
|\alpha_n(t,z)|+\frac{8 \beta_n}{\delta^2}+\frac{8}{\delta^2} \int \limits_0^t
\left( |\alpha_n(s,z)|+\frac{8\beta_n}{\delta^2} \right) e^{8 (t-s)/\delta^2}
\, ds < \eps
\end{equation}
for all $n \ge N$ and all $(t,z) \in [0,1] \times V_{\delta}$.

\smallskip

(iii) Let $z \in V_{\delta}$ and let $\sigma=\sigma_{n,\delta,z}$ be the first
time $s \ge 0$ such that $|g^{(n)}_s(z)-g_s(z)| \ge \delta/4$. 
If $0 \le t \le \min \{\sigma,1\}$, then for all $n \ge N$,
$$ |g_t^{(n)}(z)-U^{(n)}_j(t)| \ge
|g_t(z)-U_j(t)|-|g^{(n)}_t(z)-g_t(z)|-|U_j(t)-U_j^{(n)}(t)| \ge \delta/2 \, .$$

\smallskip

We are now in a position to show that $g_t^{(n)}$ converges to $g_t$ uniformly
on $[0,1] \times V_{\delta}$.
Let
$h_n(t):=g_t(z)-g^{(n)}_t(z)$. Then
\begin{eqnarray*}
& & \hspace*{-0.5cm}|h_n(t)|= |g_t(z)-g_t^{(n)}(z)|\\[1mm]
& & \hspace*{-0.2cm}
 =  \left| \int \limits_{0}^t \left[ \frac{2
      \lambda_1(s)}{g_s(z)-U_1(s)}+\frac{2
      \lambda_2(s)}{g_s(z)-U_2(s)}-\frac{2
      \lambda_1^{(n)}(s)}{g^{(n)}_s(z)-U^{(n)}_1(s)}-\frac{2
      \lambda_2^{(n)}(s)}{g^{(n)}_s(z)-U^{(n)}_2(s)} \right] \, ds \right|\\[2mm]
& & \hspace{-0.2cm} \le \left| \int \limits_{0}^t  \left( \frac{2 \left(
      \lambda_1(s)-\lambda_1^{(n)}(s) \right)}{g_s(z)-U_1(s)}+ \frac{2 \left(
      \lambda_2(s)-\lambda_2^{(n)}(s) \right)}{g_s(z)-U_2(s)} \right) \, ds
\right| \\[1mm]
 & & \hspace*{2.0cm}  +  \int \limits_0^t \left| \frac{2 \lambda_1^{(n)}(s)}{g_s(z)-U_1(s)}-
   \frac{2 \lambda_1^{(n)}(s)}{g_s^{(n)}(z)-U_1^{(n)}(s)}\right| \, ds \\[1mm]
 & & \hspace*{2.0cm} 
+ \int
 \limits_{0}^t \left| \frac{2 \lambda_2^{(n)}(s)}{g_s(z)-U_2(s)}-
   \frac{2 \lambda_2^{(n)}(s)}{g_s^{(n)}(z)-U_2^{(n)}(s)} \right| \, ds \\[2mm]
    & &  \hspace*{-0.2cm} \le |\alpha_n(t,z)| 
+ 2 \int \limits_{0}^t 
\frac{\left|h_n(s)\right|+\left|U_1(s)-U_1^{(n)}(s)\right|}{\big|g_s(z)-U_1(s)\big|
  \, \left|g_s^{(n)}(z)-U_1^{(n)}(s)\right|}
\, ds \\[1mm] & &  \hspace*{2.0cm}
+ \, 2 \int \limits_{0}^t 
\frac{\left|h_n(s,z)\right|+\left|U_2(s)-U_2^{(n)}(s)\right|}{\big|g_s(z)-U_2(s)\big|
  \, \left|g_s^{(n)}(z)-U_2^{(n)}(s)\right|}
\, ds \, .\\[1mm]
\end{eqnarray*}
Therefore, we have for all $0 \le t \le \min \{\sigma,1\}$ and every $n \ge N$
in view of of (ii) and (iii), 
\begin{eqnarray*}
|h_n(t)|   \le |\alpha_n(t,z)|+\frac{8 \beta_n}{\delta^2}+\frac{8}{\delta^2} \int
\limits_{0}^t |h_n(s)| \, ds \, .
\end{eqnarray*}
The Gronwall lemma \cite[p.~198]{FR75} shows that this estimate implies
$$ |h_n(t)| \le |\alpha_n(t,z)|+\frac{8 \beta_n}{\delta^2}+\frac{8}{\delta^2}
\int \limits_0^t \left[ |\alpha_n(s,z)|+\frac{8\beta_n}{\delta^2} \right] e^{8
  (t-s)/\delta^2} \, ds \, .$$
Hence, in view of (\ref{eq:w1}), we get $|h_n(t)|<\eps<\delta/4$ for all $0
\le t \le \min\{\sigma,1\}$ and every $n \ge N$. In particular, $\sigma \ge 1$, so we have for all
$n \ge N$
$$ |g^{(n)}_t(z)-g_t(z)|=|h_n(t)| <\eps \, , \qquad z \in V_{\delta}, \, t \in [0,1] \, .$$
This completes the proof of Theorem \ref{thm:control}.
\end{proof}

\begin{proof}[Proof of Theorem \ref{thm:aux} I: Closedness]
Let $U^{(n)}_j \in C[0,1]$ be  driving functions for Loewner
parametrizations of  subhulls of $\Theta_1 \cup \Theta_2$, and assume
that $U^{(n)}_j \to U_j$ uniformly on $[0,1]$ for $j=1,2$. Let
$\lambda^{(n)}_j \in L^1[0,1]$ be the corresponding weight functions
and $g^{(n)}_t$ the associated Loewner chains.
As the set of functions in $L^1[0,1]$ with values (a.e.) in the
interval $[0,1]$ is a weakly compact subset of $L^1[0,1]$, we can assume that $\lambda^{(n)}_j$ converges
weakly to some $\lambda_j \in L^1[0,1]$, where
$0 \le \lambda_j(t) \le 1$ and $\lambda_1(t)+\lambda_2(t)=1$ for a.e.~$t \in
[0,1]$. Let $g_t$ be the Loewner chain associated to $\lambda_1,\lambda_2$ and
$U_1, U_2$. By Theorem \ref{thm:control}, $g^{(n)}_t \to g_t$ in the
Carath\'eodory sense. Let $H^{(n)}_t$ and $H_t$ be the domains of $g^{(n)}_t$
and $g_t$. Since  the sets $\H \backslash H^{(n)}_t$  are subhulls
of $\Theta_1 \cup \Theta_2$, also $K_t:=\H \backslash H_t$  is a subhull of
$\Theta_1 \cup \Theta_2$, so $K_t=\H \backslash (\gamma_1(0,t] \cup
\gamma_2(0,t])$, where $g_t(\gamma_j(t))=U_j(t)$. Clearly,
$(\gamma_1,\gamma_2)$ is a Loewner parametrization of the subhull $K_1$
of $\Theta_1 \cup \Theta_2$ with driving functions $U_1,U_2$.
\end{proof}

\section{Capacity estimates and proof of Theorem \ref{thm:aux}.}

Let $\Theta_1, \Theta_2$ be slits with disjoint closures and 
$\hcap(\Theta_1)=\hcap(\Theta_2)=2$. 
In this section we will finish the proof of Theorem \ref{thm:aux} by showing that  the set
of driving functions for all Loewner parametrizations of subhulls of
$\Theta_1 \cup\Theta_2$ is  a precompact subset of the Banach space $C[0,1]$. 
This requires a number of technical estimates for the half--plane capacities of
two--slits and their subhulls.

\medskip

We start with the following lemma,
 which describes a number of well--known, but essential properties of half--plane capacity.
For a geometric interpretation of half--plane capacity, we refer to \cite{MR2576752,RW}.

\begin{lemma}\label{hcap}
 Let $A_1,A_2$ be hulls.
\begin{itemize}
 \item[(a)] If $A_1\cup A_2$ and $A_1\cap A_2$ are hulls, then $$\hcap(A_1)+\hcap(A_2)\geq \hcap(A_1\cup A_2)+\hcap(A_1\cap A_2).$$
\item[(b)] If  $A_1\subset A_2,$ then $\hcap(A_2)=\hcap(A_1)+\hcap(g_{A_1}(A_2\setminus A_1))\geq\hcap(A_1).$
\item[(c)] If  $A_1 \cup A_2$ is a hull and $A_1 \cap A_2=\emptyset$, then
$\hcap(g_{A_1}(A_2)) \leq \hcap(A_2).$
\end{itemize}
\end{lemma}
\begin{proof} For (a) and (b) see \cite[p.~71]{Lawler:2005}. 
Now let $A_1 \cup A_2$ be a hull such that $A_1 \cap A_2=\emptyset$.
Then  (b) implies
$\hcap(A_1)+\hcap(g_{A_1}(A_2))=\hcap(A_1\cup A_2)$, while
(a) shows $\hcap(A_1\cup A_2) \le \hcap(A_1)+\hcap(A_2)$. This proves (c).
\end{proof}

Next we prove a refinement of Lemma \ref{hcap} (c) when the hulls are slits.

\begin{lemma} \label{hcap2new}
Let $\Theta_1$ and $\Theta_2$ be slits with disjoint closures. Then
there is a constant $c>0$ such that
$$ c \le 
\frac{\hcap(B_1 \cup \Theta_2)-\hcap
    (A_1 \cup \Theta_2)}{\hcap(B_1)-\hcap(A_1)} \, $$
for all subslits $A_1 \subsetneq B_1 \subseteq \Theta_1$.
\end{lemma}

We note that a local version of Lemma \ref{hcap2new} in the sense of
$$ \lim \limits_{\hcap(B_1) \searrow \hcap(A_1)} \frac{\hcap(B_1\cup \Theta_2)-\hcap
    (A_1\cup \Theta_2)}{\hcap(B_1)-\hcap(A_1)}>0 \, \quad  \text{ for fixed
  } A_1 \, ,$$
has been proved by Lawler, Schramm and Werner \cite[Lemma 2.8]{LSW:2001}.
Our proof shows how to obtain the global statement of Lemma \ref{hcap2new} from
this local version.

\begin{proof} 
Using $g_{A_1 \cup \Theta_2}=g_{\Delta} \circ g_{\Theta_2}$ for $\Delta:=g_{\Theta_2}(A_1)$
and Lemma
\ref{hcap}, it is easy to see that
\begin{equation} \label{eq:hcap2new0}
 \hcap(B_1 \cup \Theta_2)-\hcap
    (A_1 \cup \Theta_2)=\hcap(g_{\Theta_2}(B_1))-\hcap(g_{\Theta_2}(A_1)) \, .
\end{equation}
Let $T:=\hcap(\Theta_1)/2>0$ and let $\theta : [0,T] \to \C$ be the
  parametrization of $\Theta_1$ by its half--plane capacity (see \cite[Remark 4.5]{Lawler:2005}).
 For fixed $s \in
  [0,T]$ let $L_s:=\theta (0,s]$ and $\gamma(s):=\hcap(
  g_{\Theta_2}(L_s))$. Hence, in view of (\ref{eq:hcap2new0}) and 
since $\hcap(L_s)=2s$, all we  need to show is that
  there is a constant $c>0$ such that
\begin{equation} \label{eq:hcap2new1}
c \le \frac{\gamma(s)-\gamma(\tau)}{s-\tau} \quad \text{ for all } 0 \le  \tau<s
\le T \, .
\end{equation}
In order to prove (\ref{eq:hcap2new1}), we proceed in several steps.

\smallskip

(i) \, Fix $\tau \in [0,T)$. For $s \in [\tau,T]$ let
$$\begin{array}{rcl}
& K_s:=g_{L_{\tau}}(L_s \backslash L_{\tau})  , \qquad &
K_s^*:=g_{g_{L_{\tau}}(\Theta_2)}(K_s) \\[2mm]
& b(s):=\hcap(K_s) \, , \, & b^*(s):=\hcap( K_s^*)\, . 
\end{array}$$
Then, by \cite[Lemma 2.8]{LSW:2001},
the right derivatives $\dot{b}_+(\tau)$ and $\dot{b}^*_+(\tau)$ of $b$ and $b^*$ at $\tau$
exist and
\begin{equation} \label{eq:hcap2new2}
\dot{b}^*_+(\tau)=\left[ g'_{g_{L_{\tau}}(\Theta_2)} \left(
    g_{L_{\tau}} (\theta(\tau)) \right) \right]^2 \dot{b}_+(\tau) \, .
\end{equation}
Here, $$g_{L_{\tau}}(\theta(\tau)):=\lim \limits_{z \to \theta(\tau)}
g_{L_{\tau}}(z)\, , $$ 
where the limit is taken over $z \in \H \backslash \Gamma_{\tau}$.
Now note that by Lemma \ref{hcap} b),
$$ b(s)=\hcap( g_{L_{\tau}}(L_s \backslash
L_{\tau})=\hcap(L_{s})-\hcap(L_{\tau})=2(s-\tau) \, $$
and, in a similar way, $b^*(s)=
\gamma(s)-\gamma(\tau)$.
Therefore, (\ref{eq:hcap2new2}) shows that the right derivative
$\dot{\gamma}_+(\tau)$ of the function $\gamma :[0,T] \to \R$ exists for every
$\tau \in [0,T)$ and 
\begin{equation} \label{eq:hcap2new3}
 \dot{\gamma}_+(\tau)= 2 \left[ g'_{g_{L_{\tau}}(\Theta_2)} \left(
    g_{L_{\tau}} (\theta(\tau)) \right) \right]^2 \, .
\end{equation}
(ii) \, Next  $\tau \mapsto U(\tau):=g_{L_{\tau}}(\theta(\tau))$ is
continuous on $[0,T)$ (see \cite[Lemma 4.2]{Lawler:2005}).
Furthermore, since $\Theta_1 \cap \Theta_2=\emptyset$, i.e., 
$U(\tau)\not \in g_{L_{\tau}}(\Theta_2)$, the function
$g_{g_{L_{\tau}}(\Theta_2)}$ has an analytic continuation to a
neighborhood of $U(\tau)$ and $g'_{g_{L_{\tau}(\Theta_2)}}(U(\tau))\not=0$,
see \cite[p.~69]{Lawler:2005}.
Since $\tau \mapsto g_{g_{L_{\tau}}(\Theta_2)}$ is continuous in the topology
of locally uniform convergence, we hence conclude from  (\ref{eq:hcap2new3})  that
$\dot{\gamma}_+$ is a continuous nonvanishing function  on the interval $[0,T)$. 

\smallskip

(iii) \, From (ii) we see that $\gamma : [0,T] \to \R$ is continuous, has
a right derivative $\dot{\gamma}_+(\tau)$ for every point $\tau \in [0,T)$ and
$\dot{\gamma}_+ :[0,T) \to \R$ is continuous. By Lemma 4.3 in
\cite{Lawler:2005}, $\gamma :(0,T] \to \R$ is differentiable with
$\dot{\gamma}(\tau)=\dot{\gamma}_+(\tau)$ for every $\tau \in (0,T)$. Hence
the mean value theorem shows that (\ref{eq:hcap2new1}) holds with
$$c:=\min \limits_{\tau \in [0,T]} \left[ g'_{g_{L_{\tau}}(\Theta_2)} \left(
    g_{L_{\tau}} (\theta(\tau)) \right) \right]^2 >0 \, .$$
\end{proof}

We shall need the following slight extension of Lemma \ref{hcap2new}.

\begin{lemma}\label{hcap2}
 Let  $\Theta_1$ and $\Theta_2$ be slits with disjoint closures. Then there
 exists a constant $c>0$ such that 
$$ c \le \frac{\hcap(B_1 \cup B_2)-\hcap (A_1 \cup
  A_2)} {\hcap(B_j)-\hcap(A_j)}  \, , \qquad j=1,2 \, , $$
for all subslits $A_1 \subsetneq B_1$ of $\Theta_1$ and $A_2 \subsetneq B_2$
of $\Theta_2$.
\end{lemma}

\begin{proof}
It suffices to prove the lemma for $j=1$. If we apply Lemma \ref{hcap} (b) for
the  hulls $B_1 \cup A_2 \subseteq B_1 \cup B_2$ and then Lemma \ref{hcap}
(a) for the hulls $B_1 \cup A_2$ and $A_1 \cup \Theta_2$, we obtain
\begin{eqnarray*}
 \hcap(B_1 \cup B_2)-\hcap (A_1 \cup A_2) &\ge&  \hcap (B_1 \cup A_2)-\hcap(A_1
\cup A_2) \\ &\ge & \hcap (B_1 \cup \Theta_2)-\hcap (A_1 \cup \Theta_2) \, .
\end{eqnarray*}
Therefore, the estimate of Lemma \ref{hcap2new} completes the proof of Lemma \ref{hcap2}.
\end{proof}

Let $A$ be a hull and $g_A:=h_A^{-1}$. Then $h_A$ maps $\H$ into $\H$ and
$h_A$ is
analytic at $\infty$ such that
$h_A(w)=w-\hcap(A)/w+\LandauO\left(|w|^{-2}\right)$. Hence, by the well--known
Nevanlinna representation formula (see \cite[Thm.~5.3]{RR94} or \cite{GB92}),
\begin{equation} \label{eq:nevanlinna}
h_A(w)=w+\int \limits_{\R} \frac{\mu_A(t)}{t-w} \, , 
\end{equation}
where $\mu_A$ is a finite positive measure on $\R$ with compact support
$\supp(\mu_A)$ and total mass
\begin{equation} \label{eq:hcapmeasure}
\mu_A(\R)=\hcap(A) \, .
\end{equation}

Note that by Schwarz reflection, $g_A$ has an analytic continuation across 
$\R \backslash \overline{A}$ with $g_A(\R \backslash \overline{A}) \subseteq \R$
and $h_A$ is analytic at $g_A(x)$ for every $x \in \R \backslash
\overline{A}$. The Stieltjes inversion formula (see \cite[Thm.~5.4]{RR94})
shows that $g_A(x) \not\in \supp(\mu_A)$ for every $x \in \R \backslash \overline{A}$.

\begin{lemma} \label{lem:hcapneu}
Let $A$ be a hull.
\begin{itemize}
\item[(a)] If $\overline{A} \cap \R$ is contained in the closed interval
  $[a,b]$, then  $g_A(\alpha) \le \alpha$ for every $\alpha \in \R$ with
  $\alpha<a$ and  $g_A(\beta) \ge \beta$ for every $b \in \R$ with $\beta >b$.
\item[(b)] If the open interval $(a,b)$ is contained in $\R \backslash
  \overline{A}$, then 
$|g_A(\beta)-g_A(\alpha)| \le |\beta-\alpha|$ for all $\alpha, \beta \in (a,b)$.
\end{itemize}
\end{lemma}

Hence, roughly speaking, $g_A$ is expanding outside the closed convex hull of
$\overline{A} \cap \R$ and nonexpanding in between points of $\overline{A} \cap \R$.

\begin{proof}
(a) Let $\alpha<a$. Using the Nevanlinna representation formula 
(\ref{eq:nevanlinna}) for $w=g_A(\alpha)$ yields
$$ g_A(\alpha)=h_A(g_A(\alpha))-\int \limits_{\R}
\frac{d\mu_A(t)}{t-g_A(\alpha)} \, .$$
Since the interval $(-\infty,g_A(\alpha)]$ has no point in common with
$\supp(\mu_A)$, we actually integrate over a set for which the integrand is
nonnegative, so $g_A(\alpha) \le \alpha$. The proof of $g_A(\beta) \ge \beta$
for every $\beta>b$ is similar.

\smallskip

(b) Let $\alpha,\beta \in (a,b)$ and assume $\alpha \le \beta$, so
$g_A(\alpha) \le g_A(\beta)$. If we subtract (\ref{eq:nevanlinna}) for
$w=g_A(\alpha)$ from (\ref{eq:nevanlinna}) for $w=g_A(\beta)$, then a short
computation leads to
$$ g_A(\beta)-g_A(\alpha)=\beta-\alpha+\int \limits_{\R}
\frac{g_A(\alpha)-g_A(\beta)}{(t-g_A(\alpha)) (t-g_A(\beta))} \, d\mu_A(t) \,
.$$
Since the closed interval $[g_A(\alpha),g_A(\beta)]$ is disjoint from 
$\supp(\mu_A)$, we integrate over a set for which the integrand is
nonpositive, so $0 \le g_A(\beta)-g_A(\alpha) \le \beta-\alpha$.
\end{proof}

Let $\Gamma$ be the union of two slits $\Theta_1$ and
$\Theta_2$ with disjoint closures. Then $g_{\Gamma}$  extends continuously onto each of the sides of
$\Theta_1$ and of $\Theta_2$
and maps them into $\R$. For every $c\in \Gamma$ which is neither the tip of
$\Theta_1$ nor of $\Theta_2$, we write $g^+_\Gamma(c)$ for the image w.r.t the right side and
$g^-_\Gamma(c)$ w.r.t. the left side, so that $g^-_\Gamma(c)<g^+_\Gamma(c).$

\begin{lemma} \label{lem:neu1}
Let $\Theta_1$  and $\Theta_2$ be two slits which start at $p_1 \in
\R$ resp.~$p_2 \in \R$ such that $p_1<p_2$ and $\overline{\Theta}_1 \cap \overline{\Theta}_2=\emptyset$. Then
\begin{itemize}
\item[(a)]
$ g^-_{\Theta_1\cup \Theta_2}(p_1) \le g^-_{B_1 \cup B_2}(p_1) \le g^+_{B_1
  \cup B_2}(p_2) \le  g^+_{\Theta_1\cup \Theta_2}(p_2)$, and
\item[(b)] $g^-_{B_1 \cup B_2}(p_2)-g^+_{B_1 \cup B_2}(p_1) \ge g^-_{\Theta_1
    \cup \Theta_2}(p_2)-g^+_{\Theta_1\cup\Theta_2}(p_1)$
\end{itemize}
for all subslits $B_1 \subseteq \Theta_1$ and $B_2 \subseteq \Theta_2$.
\end{lemma}

\begin{proof}
(a) Let $A_1:=g_{B_1\cup B_2}(\Theta_1\backslash B_1)$ and $A_2:=g_{B_1\cup
  B_2}(\Theta_2\backslash B_2)$.
Then $A_1$ and $A_2$ are two disjoint slits  which start say at $a \in \R$
resp.~$b \in \R$. Let $A:=A_1 \cup A_2$. Then $A$ is a hull such that
$\overline{A} \cap \R \subseteq [a,b]$. Now $\alpha:=g^-_{B_1 \cup B_2}(p_1) \le a$, so Lemma \ref{lem:hcapneu} (a) implies $g_A(\alpha) \le \alpha$. Since
$g_{\Theta_1 \cup \Theta_2}=g_A \circ g_{B_1\cup B_2}$, this shows that
$g^-_{\Theta_1\cup \Theta_2}(p_1) \le g^-_{B_1 \cup B_2}(p_1)$ and proves the
left--hand inequality. The proof of the right--hand inequality is similar.

\smallskip

(b) Let $A:=g_{B_1 \cup B_2}(\Theta_1\backslash B_1 \cup \Theta_2 \backslash
B_2)$. Then $A$ is a hull with $\overline{A}\cap \R=\{a,b\}$ such that $a < b$
and $a < g^+_{B_1\cup B_2}(p_1) \le g^-_{B_1 \cup B_2}(p_2) < b$. Hence,
$$
g^-_{\Theta_1 \cup \Theta_2}(p_2)-g^+_{\Theta_1\cup\Theta_2}(p_1) =
g_A(g^-_{B_1\cup B_2}(p_2))-g_A(g^+_{B_1\cup B_2}(p_1)) \le
g^-_{B_1\cup B_2}(p_2)-g^+_{B_1\cup B_2}(p_1)  $$
by Lemma \ref{lem:hcapneu} (b).
\end{proof}

\begin{lemma} \label{lem:II}
Let $\Theta_1$ and $\Theta_2$ be slits with disjoint closures. Then there is a
constant $L>0$ such that
$$ |g_{B_1 \cup A_2}(b_1)-g_{B_1 \cup B_2}(b_1)| \le L \cdot
|\hcap(B_2)-\hcap(A_2)|$$
for all subslits $A_2,B_2$ of $\Theta_2$ and every subslit $B_1$ of $\Theta_1$
with tip $b_1 \in B_1$.
\end{lemma}

\begin{proof}
We assume $A_2 \subseteq B_2$. Then $A:=g_{B_1 \cup A_2}(B_2 \backslash A_2)$
is a hull, so  Lemma \ref{hcap} (b) shows $\hcap(A)=\hcap(B_1 \cup
B_2)-\hcap(B_1 \cup A_2)$. On the other hand, 
 Lemma \ref{hcap} (a) applied to the two hulls $B_2$ and $B_1 \cup A_2$ 
gives
$\hcap(B_1 \cup B_2)-\hcap(B_1 \cup A_2) \le
\hcap(B_2)-\hcap(A_2)$. Therefore, 
\begin{equation} \label{eq:II1}
 \hcap(A) \le \hcap(B_2)-\hcap(A_2)
\, .
\end{equation}
Note that $g_{B_1 \cup B_2}=g_A \circ g_{B_1 \cup A_2}$, so the Nevanlinna
representation formula (\ref{eq:nevanlinna}) for $h_A:=g_A^{-1}$ and $w=g_{B_1
  \cup B_2}(b_1)$ shows that
\begin{equation} \label{eq:II2}
g_{B_1 \cup A_2}(b_1)-g_{B_1 \cup B_2}(b_1)=h_A(g_{B_1 \cup B_2}(b_1))-g_{B_1
  \cup B_2}(b_1)=\int \limits_{\R} \frac{d\mu_A(t)}{t-g_{B_1 \cup B_2}(b_1)}
\, .
\end{equation}
Let $\Theta_1$ start at $p_1 \in \R$ and $\Theta_2$ start at $p_2 \in \R$ with
$p_1<p_2$. Then the interval $(-\infty,g^-_{B_1 \cup B_2}(p_2)]$ is disjoint
from the support $\supp(\mu_A)$ of the measure $\mu_A$, so for every $t \in
\supp(\mu_A)$, we have
$$ t-g_{B_1\cup B_2}(b_1) \ge g^-_{B_1 \cup B_2}(p_2)-g^+_{B_1 \cup B_2}(p_1)
\ge  g^-_{\Theta_1 \cup \Theta_2}(p_2)-g^+_{\Theta_1 \cup \Theta_2}(p_1)=:L^{-1}>0$$
by Lemma \ref{lem:neu1} (b).
Hence (\ref{eq:II2}) leads to
$$ 0< g_{B_1 \cup A_2}(b_1)-g_{B_1 \cup B_2}(b_1) \le L \int \limits_{\R} d\mu_A(t)
=L \hcap(A)
$$
by (\ref{eq:hcapmeasure}). In view of (\ref{eq:II1}) the proof of Lemma \ref{lem:II}
is complete.
\end{proof}

\begin{lemma} \label{lem:III}
Let $\Theta_1$ and $\Theta_2$ be slits with disjoint closures. Then there
exists a monotonically increasing function $\omega : [0,\hcap(\Theta_1)] \to
[0,\infty)$ with $\lim \limits_{\delta \searrow 0} \omega(\delta)=\omega(0)=0$ such that
\begin{equation} \label{eq:III1}
 |g_{A_1 \cup A_2}(a_1)-g_{B_1 \cup A_2}(b_1)| \le \omega \left(
  |\hcap(A_1)-\hcap(B_1)| \right)
\end{equation}
for all subslits $A_1$ and $B_1$ of $\Theta_1$ with tips $a_1 \in A_1$ and
$b_1 \in B_1$ and every subslit $A_2 \subseteq \Theta_2$.
\end{lemma}

\begin{proof}
We first define $\omega(\delta)$ for $\delta \in (0,\hcap(\Theta_1)]$ by
$$ \omega(\delta):=\sup \{ g^+_{B_1}(a_1)-g^-_{B_1}(a_1) \} \, . $$
Here the supremum is taken over all subslits $A_1 \subseteq B_1$ of
$\Theta_1$ such that $\hcap(B_1)-\hcap(A_1) \le \delta$ and $a_1$ is the tip
of $A_1$. Clearly,  $\omega : (0,\hcap(\Theta_1)] \to (0,\infty)$ is
monotonically increasing and we need to prove (i) the estimate (\ref{eq:III1})
and (ii) $\lim_{\delta\searrow 0} \omega(\delta)=0$.

\smallskip

(i) Assume $A_1 \subseteq B_1$. Consider the slit $A:=g_{A_1 \cup A_2}(B_1
\backslash A_1)$, which starts at $g_{A_1 \cup A_2}(a_1)$. Then $g_{B_1 \cup B_2}=g_A \circ g_{A_1 \cup A_2}$, so Lemma
\ref{lem:hcapneu} (a) implies $g^{-}_{B_1 \cup A_2}(a_1) =g^-_A(g_{A_1 \cup A_2}(a_1))
\le g_{A_1 \cup
  A_2}(a_1) \le g^+_A(g_{A_1 \cup A_2}(a_1))=g^+_{B_1 \cup A_2}(a_1)$. Since
we clearly also have $g^-_{B_1 \cup A_2}(a_1) \le
g_{B_1 \cup A_2}(b_1) \le g^+_{B_1 \cup A_2}(a_1)$, we deduce
$$ |g_{A_1 \cup A_2}(a_1)-g_{B_1 \cup A_2}(b_1)| \le g^+_{B_1 \cup
  A_2}(a_1)-g^-_{B_1 \cup A_2}(a_1) \, .$$
Since $g_{B_1 \cup A_2}=g_{g_{B_1}(A_2)} \circ g_{B_1}$, Lemma \ref{lem:neu1}
(b) shows that
$$ g^+_{B_1 \cup A_2}(a_1)-g^-_{B_1 \cup A_2}(a_1)=g_{g_{B_1}(A_2)} \left(
  g^+_{B_1}(a_1) \right)-g_{g_{B_1}(A_2)} \left( g_{B_1}(a_1) \right) \le
g^+_{B_1}(a_1)-g^-_{B_1}(a_1) \, ,$$
so we get $ |g_{A_1 \cup A_2}(a_1)-g_{B_1 \cup A_2}(b_1)| \le
\omega(\hcap(B_1)-\hcap(A_1))$, i.e.~the estimate (\ref{eq:III1}) holds.

\smallskip

(ii) Let $c_1:=\hcap(\Theta_1)/2$, denote by $\theta_1 : [0,c_1] \to \C$
the parametrization of $\Theta_1$ by its half--plane capacity and let $U :
[0,c_1] \to \R$ be the driving function for the slit $\Theta_1$ according to
Theorem \ref{slitex}.
Let $A_1 \subseteq B_1$ be subslits of $\Theta_1$ and let $a_1$ be the tip of
$\Theta_1$. Then there are $t,s \in [0,c_1]$ with $t \le s$ such that
$\theta_1(t)=a_1$, $\theta_1(0,s]=B_1$ and $s-t=\hcap(B_1)/2-\hcap(A_1)/2$.
Consider the slit $P:=g_{A_1}(B_1 \backslash A_1)$, so $\overline{P} \cap
\R=\{U(t)\}$ and $g^+_{B_1}(a_1)-g^-_{B_1}(a_1)$ is the euclidean length of
the interval $g_P(P)$.
 By Remark 3.30 in \cite{Lawler:2005} there is an absolute
constant $M>0$ such that
\begin{equation} \label{eq:III2}
g^+_{B_1}(a_1)-g^-_{B_1}(a_1) \le M \cdot \diam(P) \, ,
\end{equation}
where $\diam(P):=\sup \{|p-q| \, : \, p,q \in P\}$. Define 
$\operatorname{rad}(P):=\{|z-U(t)| \, : \, z \in P\}$. Then, by Lemma 4.13 in
\cite{Lawler:2005},
\begin{eqnarray*}
 \operatorname{rad}(P) &\le & 4\max \left\{ \sqrt{s-t}, \sup \limits_{t \le \tau \le s}
  |U(\tau)-U(t)| \right\} \\[3mm]
& \le & 4 \max \left\{ \sqrt{s-t}, \sup \limits_{|\tau-\sigma|\le s-t} 
  |U(\tau)-U(\sigma)| \right\}
 \, .
\end{eqnarray*}
Hence, if we define 
$$ \varrho(\delta):=\max \left\{ \sqrt{\delta/2}, \sup \limits_{|\tau-\sigma|\le \delta/2} 
  |U(\tau)-U(\sigma)| \right\} $$
for $\delta \in [0,\hcap(\Theta_1)]$ then $\operatorname{rad}(P) \le 4 \varrho(\hcap(B_1)-\hcap(A_1))$.
Using the obvious estimate $\diam(P) \le 2 \operatorname{rad}(P)$, we obtain
from (\ref{eq:III2}) that  $g^+_{B_1}(a_1)-g^-_{B_1}(a_1) \le 8 M
\varrho(\hcap(B_1)-\hcap(A_1))$. Recalling the definition of $\omega(\delta)$,
this shows that $\omega(\delta) \le 8 M \varrho(\delta)$ for all $\delta \in
(0,\hcap(\Theta_1)]$. Since the continuous driving function $U : [0,c_1] \to
\R$ is uniformly continuous on $[0,c_1]$, we see that $\varrho(\delta) \to 0$
as $\delta \searrow 0$, so $\lim \limits_{\delta \searrow 0} \omega(\delta)=0$.
\end{proof}

\begin{lemma} \label{lem:IV}
Let $\Theta_1$ and $\Theta_2$ be slits with disjoint closures. Then there
exist constants $c,L>0$ and a monotonically increasing function $\omega : [0,\hcap(\Theta_1)] \to
[0,\infty)$ with $\omega(0)=0$ such that
\begin{eqnarray*}
|g_{A_1 \cup A_2}(a_1)-g_{B_1 \cup B_2}(b_1)| & \le &
\omega \left( \frac{1}{c} \, |\hcap(A_1\cup A_2)-\hcap(B_1 \cup B_2)|
\right)  \\ & & \hspace*{1cm} + \frac{L}{c} \, |\hcap(A_1 \cup A_2)-\hcap(B_1\cup B_2)| 
\end{eqnarray*}
for all subslits $A_1$ and $B_1$ of $\Theta_1$ with
tips $a_1 \in A_1$ and $b_1 \in B_1$ and all subslits
 $A_2, B_2$ of $\Theta_2$.
\end{lemma}

\begin{proof}
We can assume $A_1 \subsetneq B_1$ and $A_2 \subsetneq B_2$.
Then 
\begin{eqnarray*}
|g_{A_1 \cup A_2}(a_1)-g_{B_1 \cup B_2}(b_1)| & \le & |g_{A_1 \cup
  A_2}(a_1)-g_{B_1 \cup A_2}(b_1)|+|g_{B_1 \cup A_2}(b_1)-g_{B_1 \cup
    B_2}(b_1)| \\
&\le &  \omega (\hcap(B_1)-\hcap(A_1))+L \left( \hcap(B_2)-\hcap(A_2) \right)
\, .
\end{eqnarray*}
by Lemma \ref{lem:III} and Lemma \ref{lem:II}. Now the estimate of Lemma
\ref{hcap2}
completes the proof of Lemma \ref{lem:IV}.
\end{proof}

\begin{proof}[Proof of Theorem \ref{thm:aux} II: Precompactness]
Let $(\gamma_1,\gamma_2)$ be a Loewner parametrization of a subhull of
$\Theta_1 \cup \Theta_2$, let $g_t:=g_{\gamma_1(0,t] \cup \gamma_2(0,t]}$ , and let 
$U_1(t)=g_{t}(\gamma_1(t))$ and $U_2(t)=g_{t}(\gamma_2(t))$
be the driving functions for $(\gamma_1,\gamma_2)$. Then
$$ g^-_{t}(\gamma_1(0)) \le U_j(t) \le
g^+_{t}(\gamma_2(0)) \, ,$$
so Lemma \ref{lem:neu1} (a) implies
$$ g^-_{\Theta_1 \cup \Theta_2}(\theta_1(0)) \le U_j(t) \le
g^+_{\Theta_1 \cup \Theta_2}(\theta_2(0)) \, , \qquad j=1,2\, .$$
This gives a uniform bound for $U_1(t)$ and $U_2(t)$.
Since $\hcap(\gamma_1(0,t] \cup \gamma_2(0,t])=2t$,
Lemma \ref{lem:IV} implies
$$
 |U_1(t)-U_1(s)|=|g_{t}(\gamma_1(t))-g_{s}(\gamma_1(s))| \le 
\omega \left( \frac{2|t-s|}{c} \right) +\frac{2L}{c} |t-s|$$
for all $t,s \in [0,1]$. This shows that the driving functions $U_1$ 
for all Loewner parametri\-zations $(\gamma_1, \gamma_2)$ are
uniformly equicontinuous on $[0,1]$. By switching the roles of $U_1$ and
$U_2$, the same result holds for the driving functions $U_2$. An application
 of the Arzel\`a--Ascoli theorem completes the proof of Theorem \ref{thm:aux}.
\end{proof}

\section{\label{existence} Proof of Theorem \ref{Charlie2}, Part I
  (Existence)}

\begin{lemma}  \label{lem:stepfunctions}
Let $\Theta_1$ and $\Theta_2$ be slits with disjoint closures and
$\hcap(\Theta_1)=\hcap(\Theta_2)=2$, and let $\alpha : [0,1] \to \{0,1\}$ be a
step function. Then there exists a Loewner parametrization $(\gamma_{1,\alpha},\gamma_{2,\alpha})$
of a subhull $A_{\alpha}$ of $\Theta_1 \cup \Theta_2$ such that for the
corresponding weight functions $\lambda_{1,\alpha}, \lambda_{2, \alpha}$
and every $t \in [0,1]$,
$$ \lambda_1(t)=1 \text{ iff } \alpha(t)=1 \quad \text{ and } \quad \lambda_2(t)=1 \text{
  iff } \alpha(t)=0 \, .$$
Moreover,  $g_{t,\alpha}:=g_{\gamma_{1,\alpha}(0,t] \cup
  \gamma_{2,\alpha}(0,t]}$ is the solution to the Loewner equation
\begin{equation} \label{eq:bang}
\begin{array}{rcl}
\dot{g}_{t,\alpha}(z)&=&\displaystyle \frac{2\alpha(t)}{g_{t,\alpha}(z)-U_{1,\alpha}(t)}+\frac{2(1-\alpha(t))}{g_{t,\alpha}(z)-U_{2,\alpha}(t)}\,
, \quad t \in [0,1] \, , \\[2mm]
   g_{0,\alpha}(z)&=& z \, ,
\end{array}
\end{equation}
where $U_{1,\alpha}(t)=g_{t,\alpha}(\gamma_{1,\alpha}(t))$ and
$U_{2,\alpha}(t)=g_{t,\alpha}(\gamma_{2,\alpha}(t))$ are the continuous
driving functions of the Loewner parametrization $(\gamma_{1,\alpha},\gamma_{2,\alpha})$.
\end{lemma}

\begin{proof}
 For $j=1,2$ let $\theta_j : [0,1] \to \Theta_j$ be the 
parametrization of $\Theta_j$ by its half--plane capacity. 
We construct two monotonically increasing continuous functions $x_{1,\alpha}, x_{2,\alpha} :
[0,1] \to [0,1]$ such that for every
subinterval $I \subseteq [0,1]$,
\begin{itemize}
\item[(I)] 
$x_{1,\alpha}$ is constant on $I$ if and only if
$\alpha|_{I} \equiv 0$, 
\item[(II)] $x_{2,\alpha}$  is constant on $I$ if and only
  if $\alpha|_I \equiv 1$,
and
\item[(III)] $(\gamma_{1,\alpha},\gamma_{2,\alpha}):=(\theta_1 \circ
  x_{1,\alpha},\theta_2 \circ x_{2,\alpha})$ defines a Loewner parametrization of a subhull of
$\Theta_1 \cup \Theta_2$,
\end{itemize}
as follows.
Let $0=\tau_0 < \tau_1<\ldots<\tau_N=1$ be a partition of
$[0,1]$ into subintervals $I_j:=[\tau_{j-1},\tau_j)$. We may assume that
$\alpha \equiv 1$ on $I_1\cup I_3 \cup I_5 \cup \ldots$ and $\alpha \equiv 0$
on $I_2 \cup I_4 \cup I_6\cup \ldots$. We construct $\gamma_{1,\alpha},
\gamma_{2,\alpha}$ on the closure $\overline{I}_j$ by induction.

\medskip

(i) For $t \in I_1$ let $x_{1,\alpha}(t):=2 t$ and $x_{2,\alpha}(t):=0$.\\[1mm]
(ii) Assume that $x_{1,\alpha}, x_{2,\alpha}$ have been constructed on $\overline{I}_{j-1}$.
Consider the case $\alpha|_{I_{j}} \equiv 0$, so $\alpha|_{I_{j-1}}\equiv
1$. Then $x_{2,\alpha}(t)=x_{2,\alpha}(\tau_{j-1})$ for every $t \in
I_{j-1}$. Now fix $t \in (\tau_{j-1},\tau_j]$. Let
$x_{1,\alpha}(t):=x_{1,\alpha}(\tau_{j-1})$. Clearly, there exists a unique
$c \in[ x_{2,\alpha}(\tau_{j-1}),1]$ such that $ \hcap (\theta_1(0,x_{1,\alpha}(\tau_{j-1})] \cup
\theta_2(0,c])=2t$. Let $x_{2,\alpha}(t):=c$. \\[1mm]
By construction, $x_{1,\alpha},x_{2,\alpha}$ satisfy (I)--(III), so
$(\gamma_{1,\alpha}, \gamma_{2,\alpha})=(\theta_1 \circ x_{1,\alpha},\theta_2
\circ x_{2,\alpha})$ is a Loewner parametrization of a subhull $A_{\alpha}$ of
$\Theta_1 \cup \Theta_2$
such that $\lambda_{1,\alpha}(t)=\alpha(t)$ and
$\lambda_{2,\alpha}(t)=1-\alpha(t)$ for all $t \in [0,1]$.

\medskip

It remains to show that $g_{t, \alpha}=g_{\gamma_{1,\alpha}(0,t] \cup
  \gamma_{2,\alpha}(0,t]}$ is the solution of the Loewner equation
(\ref{eq:bang}). We again proceed by induction and first prove this for $t \in \overline{I}_1=[0,\tau_1]$.
Note that  for $t \in \overline{I}_1$ we have $\alpha(t)= 1$ and  
the Loewner parametrization $(\gamma_{1,\alpha},
\gamma_{2, \alpha})$ generates the \textit{one--slit}
$\gamma_{1,\alpha}(0,\tau_0] \cup
\gamma_{2,\alpha}(0,\tau_0]=\gamma_{1,\alpha}(0,\tau_0]$, so
 $g_{t,\alpha}=g_{\gamma_{1,\alpha}(0,t]}$ and $g_{t,\alpha}(\gamma_{1,\alpha}(t))=U_{1,\alpha}(t)$.
 Hence, by Theorem
\ref{satz:1}, we have
$$ \dot{g}_{t,\alpha}(z)=\frac{2}{g_{t,\alpha}(z)-U_{1,\alpha}(t)}=
\frac{2 \alpha(t)}{g_{t,\alpha}(z)-U_{1,\alpha}(t)}+\frac{2 (1-\alpha(t))}{g_{t,\alpha}(z)-U_{2,\alpha}(t)}
 \, , \qquad t \in \overline{I}_1 \, .$$
Next assume that we  already know that $g_{t,\alpha}$ is the solution to
(\ref{eq:bang}) on the closure  of the intervall $I_1 \cup \ldots \cup
I_{j-1}$ for some $j \in \{2, \ldots, N\}$. Let $B:=\gamma_{1,\alpha}(0,\tau_{j-1}] \cup
\gamma_{2,\alpha}(0,\tau_{j-1}]$ and let $B'_t:=\gamma_{1,\alpha}(0,t] \cup
\gamma_{2,\alpha}(0,t]$ for $t \in \overline{I}_j$. Since $\alpha|_{I_j} \equiv 0$, the set
$\Gamma_t:=g_B(B'_t\backslash B)=g_B(\gamma_{2,\alpha}(\tau_{j-1},t])$ is a one--slit with parametrization
$s \mapsto g_B(\gamma_{2,\alpha}(s))$, $s  \in [\tau_{j-1},t]$.
 Lemma \ref{hcap} implies that
$\hcap (\Gamma_t)=\hcap (B'_t)-\hcap(B)=2(t-\tau_{j-1})$, so $g_B \circ
\gamma_{2,\alpha}$ is the parametrization of $\Gamma_t$ with respect to
half--plane capacity (on the interval $[\tau_{j-1},t]$). 
Hence, by Theorem \ref{satz:1}, the function $\tilde{g}_t:=g_{\Gamma_t}$ is
the solution of the one--slit equation
$$ \dot{\tilde{g}}_t(z)=\frac{2}{\tilde{g}_t(z)-\tilde{U}(t)} \, , \qquad
t\in \overline{I}_j \, ,$$
where $\tilde{U}(t)=\tilde{g}_t(g_B(\gamma_2(t))$.
Now note that
$g_{t,\alpha}=g_{\Gamma_t} \circ g_B=\tilde{g}_t \circ g_B$,
so $\tilde{U}(t)=g_{t,\alpha}(\gamma_{2,\alpha}(t))$. Therefore, using again
$\alpha|_{I_j} \equiv 0$, we get
$$ \dot{g}_{t,\alpha}(z)=\frac{2}{g_{t,\alpha}(z)-U_{2,\alpha}(t)}
=\frac{2 \alpha(t)}{g_{t,\alpha}(z)-U_{1,\alpha}(t)}+\frac{2 (1-\alpha(t))}{g_{t,\alpha}(z)-U_{2,\alpha}(t)}
 \, , \qquad t \in \overline{I}_j \, .$$
This completes the proof of Lemma \ref{lem:stepfunctions}.
\end{proof}

\begin{proof}[Proof of Theorem \ref{Charlie2} (Existence)]
Let $(\Gamma_1, \Gamma_2)$ be a two--slit with
$\hcap(\Gamma_1 \cup \Gamma_2)=2$. We choose slits $\Theta_1 \supseteq
\Gamma_1$,
  $\Theta_2\supseteq \Gamma_2$ with disjoint closures and
$\hcap(\Theta_1)=\hcap(\Theta_2)=2$.

\medskip

\textit{Step 1:} Let $\alpha : [0,1] \to \{0,1\}$ be a step function.
Using the Loewner parametrization of Lemma \ref{lem:stepfunctions}
and the associated driving functions $U_{1,\alpha}, U_{2,\alpha} \in C[0,1]$, 
 we see that the solution $g_{t,\alpha}$ of the Loewner equation
$$
\dot{g}_{t,\alpha}(z)=\frac{2\alpha(t)}{g_{t,\alpha}(z)-U_{1,\alpha}(t)}+\frac{2(1-\alpha(t))}{g_{t,\alpha}(z)-U_{2,\alpha}(t)}\,
,
  \quad g_{0,\alpha}(z)=z \, ,$$
has the property that $g_{1,\alpha}=g_{A_{\alpha}}$.

\medskip

\textit{Step 2:} By Theorem \ref{thm:aux}, the set
$\{U_{1,\alpha}, U_{2,\alpha} \, |
\, \alpha : [0,1] \to\{0,1\} \text{ step function}\}$ is a precompact subset
of $C[0,1]$.

\medskip

\textit{Step 3:} Fix $n \in \N$ and $\mu \in [0,1]$.
Let  $\alpha_{n,\mu}:[0,1]\to \{0,1\}$ be the step function defined by \begin{equation}\label{Pete} \alpha_{n,\mu}(t)=\begin{cases}
1 \quad \text{when} \quad t\in (\frac{k}{2^n}, \frac{k+\mu}{2^n}),\; k\in\{0,...,2^n-1\},\\[1mm]
0 \quad \text{when} \quad t\in (\frac{k+\mu}{2^n}, \frac{k+1}{2^n}),\; k\in\{0,...,2^n-1\}.                                                                  \end{cases}
\end{equation}
By Step 1, we find continuous driving functions $U_{1,n,\mu},
U_{2,n,\mu} : [0,1] \to \R$ such that the solution $g_{t,n,\mu}$ to
\begin{equation}\label{e1}
  \dot{g}_{t,n,\mu}=\frac{2\alpha_{n,\mu}(t)}{g_{t,n}-U_{1,n,\mu}(t)}+\frac{2(1-\alpha_{n,\mu}(t))}{g_{t,n}-U_{2,n,\mu}(t)},
  \quad g_{0,n,\mu}(z)=z. \end{equation}
for $t=1$ produces the subhull $A_{\alpha_{n,\mu}}$ of $\Theta_1 \cup
\Theta_2$. If we denote by $\theta_j : [0,1] \to \Theta_j$ the
parametrization of $\Theta_j$ by its half--plane capacity, then we can write
$A_{\alpha_{n,\mu}}=\theta_1(0,x_{1,n,\mu}]\cup
\theta_2(0,x_{2,n,\mu}]$ with $x_{1,n,\mu}, x_{2,n,\mu} \in [0,1]$.
 For fixed $n \in \N$,
$\mu \mapsto x_{1,n,\mu}$ is clearly continuous on $[0,1]$ with $x_{1,n,0}=0$ and $x_{1,n,1}=1$.
Hence, the intermediate value theorem shows that there is a number $\mu_n \in [0,1]$
such that $x_{1,n,\mu_n}=\hcap(\Gamma_1)$. Since
$\hcap(A_{\alpha_{n,\mu}})=2=\hcap(\Gamma_1 \cup \Gamma_2)$,
we get $x_{2,n,\mu_n}=\hcap(\Gamma_2)$ from Lemma \ref{hcap} (b).
 Hence, if we
set $\alpha_n:=\alpha_{n,\mu_n}$, we have $A_{\alpha_n}=\Gamma_1 \cup \Gamma_2$.
\medskip

Since $(\mu_n)$ is a sequence of real numbers in the interval $[0,1]$, we can find a subsequential limit
$\lambda:=\lim_{k \to \infty} \mu_{n_k}$. We claim that the step functions
$\alpha_{n_k}$ converge weakly in $L^1[0,1]$ to the constant function
$\lambda$. For this purpose, it suffices (see \cite[p.~118]{Jur97}) to prove that
$$  \int \limits_{a}^b \alpha_{n_k}(s) \, ds \to \lambda (b-a) \, $$
for all $0 \le a<b \le 1$, a fact which can be easily verified directly using the
definition of the step functions $\alpha_n$.

\medskip

\textit{Step 4:} If $(\alpha_{n_k})$ is the weakly convergent sequence of Step
3, we can assume with the help of Step 2 that the driving functions
$U_{1,\alpha_{n_k}}, U_{2,\alpha_{n_k}}$ converge uniformly on $[0,1]$ to
functions $U_1, U_2 \in C[0,1]$. If $g_t$ denotes the solution to the
Schramm--Loewner equation
$$
\dot{g}_t(z)=\frac{2\lambda}{g_t(z)-U_{1}(t)}+\frac{2(1-\lambda)}{g_t(z)-U_{2}(t)},\quad
g_0(z)=z \, ,$$
then by Theorem \ref{thm:control} the Loewner chains $g_{t,\alpha_{n_k}}$
converge to $g_t$ in the Carath\'eodory sense. In particular, $g_1=g_{\Gamma_1
  \cup \Gamma_2}$. This completes the proof of Theorem \ref{Charlie2} (Existence).
\end{proof}

\section{Dynamic interpretation of constant weights}

Let $\Gamma_1$ and $\Gamma_2$ be slits with disjoint closures. We have proved
in Section 4 that there exists a constant $\lambda \in [0,1]$ and driving functions
$U_1,U_2 \in C[0,1]$ such that the solution $g_t$ to the Schramm--Loewner equation
\begin{equation}\label{eq:2slits}
  \dot{g}_t(z)=\frac{2\lambda}{g_t(z)-U_{1}(t)}+\frac{2(1-\lambda)}{g_t(z)-U_{2}(t)},\quad
  g_0(z)=z \, , \end{equation}
satisfies $g_1=g_{\Gamma_1 \cup \Gamma_2}$. Let $\gamma_1(t)$ and
$\gamma_2(t)$ be the tip of the part of $\Gamma_1$ and $\Gamma_2$ respectively
at time $t$, so $(\gamma_1,\gamma_2)$ is a Loewner parametrization of
$(\Gamma_1,\Gamma_2)$ with  constant weights $\lambda$ and $1-\lambda$.
 In this section, we will derive some properties of this  Loewner parametrization
 $(\gamma_1,\gamma_2)$ and start with a simple estimate for the imaginary part of the slits.

\begin{lemma}\label{imestim} The Loewner parametrization $(\gamma_1,\gamma_2)$ satisfies
 $$\max_{z\in \gamma_1[0,t]\cup\gamma_2[0,t]}\Im z\leq 2\sqrt{t}.$$
\end{lemma}

\begin{proof} 
For fixed $t\in[0,1]$ we consider the backward Loewner
equation \begin{equation}\label{2slitsb}
  \dot{h}_s(z)=\frac{-2\lambda}{h_s(z)-U_{1}(t-s)}+\frac{-2(1-\lambda)}{h_s(z)-U_{2}(t-s)},\quad
  h_0(z)=z\in\Ha \, ,  \end{equation}
for $s \in [0,t]$. In view of (\ref{eq:2slits}) we see that
 $h_t=g_t^{-1}$, so it suffices to prove $\Im h_t(x_0) \le 2\sqrt{t}$ for each
 $x_0 \in \R$.
 Let $x_0\in\R$. For $y_0>0$ we write $h_s(x_0+iy_0)=x_s+iy_s.$ Then (\ref{2slitsb}) gives 
 \begin{eqnarray*} \dot{y}_s&=&\frac{2\lambda y_s}{(x_s-U_1(t-s))^2+y_s^2}+\frac{2(1-\lambda) y_s}{(x_s-U_2(t-s))^2+y_s^2}\\[2mm]
&\leq&   \frac{2\lambda y_s}{y_s^2}+\frac{2(1-\lambda) y_s}{y_s^2}=\frac{2}{y_s}.\end{eqnarray*}
Thus $y_s\leq \sqrt{4s+y_0^2}.$ Letting $s\nearrow t$ and then $y_0\searrow 0$ shows $\Im h_t(x_0)
\leq 2\sqrt{t}.$
\end{proof}

In the following lemma we let $\mathcal{B}(z,r):=\{w\in\C\with|z-w|<r\},$
where $z\in\C, r>0$ and for $A\subseteq \C$ we define
$\displaystyle\diam(A):=\sup_{z,w\in A}|z-w|.$ 

\begin{lemma}\label{asyhcap} Let $x(t)=\hcap(\gamma_1(0,t])$ and
$y(t)=\hcap(\gamma_2(0,t])$. Then
 $$x(t)+y(t)-2t=\Landauo(t) \quad \text{for} \quad t\to0.$$
\end{lemma}
\begin{proof}
First, we note that $x(t)+y(t)-2t \geq 0$ for all $t$ because of Lemma
\ref{hcap} a).

\medskip

We will use a formula which translates the half--plane capacity of an arbitrary hull $A$ into an expected value of a random variable derived from a Brownian motion hitting this hull. Let $B_s$ be a Brownian motion started in $z\in \Ha\setminus A.$ We write $\operatorname{\bf P}^{z}$ and $\operatorname{\bf E}^{z}$ for probabilities and expectations derived from $B_s.$ Let $\tau_A$ be the smallest time $s$ with $B_s\in \R\cup A.$ Then formula (3.6) of Proposition 3.41 in \cite{Lawler:2005} tells us
$$\hcap(A)=\lim_{y\to\infty}y\operatorname{\bf E}^{yi}[\operatorname{Im}(B_{\tau_A})].$$
Let $\varrho_t=\tau_{\gamma_1[0,t]}$ and $\sigma_t=\tau_{\gamma_2[0,t]}.$ Then we have 
\begin{eqnarray*}x(t)+y(t)-2t=\lim_{y\to \infty}y \left(\operatorname{\bf E}^{yi}[\operatorname{Im}(B_{\varrho_t});\sigma_t<\varrho_t]+ \operatorname{\bf E}^{yi}[\operatorname{Im}(B_{\sigma_t});\sigma_t>\varrho_t]\right).
\end{eqnarray*}
We will estimate the two expected values. First, $\Im(B_{\varrho_t})\leq 2\sqrt{t}$ by Lemma \ref{imestim} and we get
$$\operatorname{\bf E}^{yi}[\operatorname{Im}(B_{\varrho_t});\sigma_t<\varrho_t]\leq 2\sqrt{t} \cdot \operatorname{\bf P}^{yi}\{B_{\varrho_t}\in \gamma_1[0,t];\sigma_t<\varrho_t\}. $$
Now for $t$ small enough there exists $R>0$ such that 
$$ \gamma_1[0,s]\subset \mathcal{B}(\Re(\gamma_1(s)), R), \qquad \gamma_2[0,s]\subset \mathcal{B}(\Re(\gamma_2(s)), R),  $$
$$ \gamma_1[0,t] \cap \mathcal{B}(\Re(\gamma_2(s)), R)=\emptyset,\qquad \gamma_2[0,t] \cap \mathcal{B}(\Re(\gamma_1(s)), R)=\emptyset, $$
for all $s\in[0,t].$

\medskip

 A Brownian motion satisfying $\sigma_t<\varrho_t$ will hit $\gamma_2[0,t]$, say at $\gamma_2(s)$ for some $s\in[0,t]$, and has to leave $\mathcal{B}(\Re(\gamma_2(s)),R)\cap \overline{\Ha}$ without hitting the real axis, see Figure 1. Call the probability of this event $p_s$. Then we have
$$ \operatorname{\bf P}^{yi}\{B_{\varrho_t}\in \gamma_1[0,t];\sigma_t<\varrho_t\} \leq  \operatorname{\bf P}^{yi}\{B_{\sigma_t}\in \gamma_2[0,t]\} \cdot \sup_{s\in[0,t]}p_s. $$
Lemma \ref{imestim} implies $\Im(B_{\sigma_t})\leq 2\sqrt{t}$ and Beurling's estimate (Theorem 3.76 in \cite{Lawler:2005}) says that there exists $c_1>0$ (depending on $R$ only) such that $$p_s\leq c_1 \cdot 2\sqrt{t}.$$
(Note that Theorem 3.76 in \cite{Lawler:2005} gives an estimate on the probability that a Brownian motion started in $\D$ will not have hit a fixed curve, say $[0,1]$, when leaving $\D$ the first time. The estimate we use can be simply recovered by mapping the half-circle $\D\cap \Ha$ conformally onto $\D\setminus [0,1]$ by $z\mapsto z^2.$)\\
 We get the same estimates for $\operatorname{\bf
   E}^{yi}[\operatorname{Im}(B_{\sigma_t});\sigma_t>\varrho_t]$ and putting
 all this together gives the following upper bound for $x(t)+y(t)-2t$
\begin{eqnarray*} 
&\displaystyle   \lim_{y\to \infty}y \left(2\sqrt{t}\cdot c_1\cdot 2\sqrt{t} \cdot \operatorname{\bf P}^{yi}\{B_{\sigma_t}\in \gamma_2[0,t]\}+ 2\sqrt{t}\cdot c_1\cdot 2\sqrt{t} \cdot \operatorname{\bf P}^{yi}\{B_{\varrho_t}\in \gamma_1[0,t]\}\right)\\
 & \displaystyle \hspace*{-3.2cm} =4c_1t\cdot\displaystyle\lim_{y\to \infty}y
 \left( \operatorname{\bf P}^{yi}\{B_{\sigma_t}\in
   \gamma_2[0,t]\}+\operatorname{\bf P}^{yi}\{B_{\varrho_t}\in
   \gamma_1[0,t]\}\right)\, .\end{eqnarray*}
Here the  limit exists and (see \cite[p.~74]{Lawler:2005})
\begin{eqnarray*}
 \lim_{y\to \infty}y \operatorname{\bf P}^{yi}\{B_{\sigma_t}\in
 \gamma_2[0,t]\} &\leq & c_2  \diam(\gamma_2[0,t]) \, , \\
\lim_{y\to \infty}y \operatorname{\bf P}^{yi}\{B_{\varrho_t}\in
\gamma_1[0,t]\} &\leq & c_2 \diam(\gamma_1[0,t]) \end{eqnarray*}
with a universal constant $c_2>0$.
Finally $\diam(\gamma_j[0,t])\to 0$ for $t\to 0$ and $j=1,2;$ see, e.g., Lemma 4.13 in \cite{Lawler:2005}. Hence we have shown $x(t)+y(t)-2t=\Landauo(t).$ 
\end{proof}

\begin{figure}[h] \label{Brownian_motion}
    \centering
   \includegraphics[width=140mm]{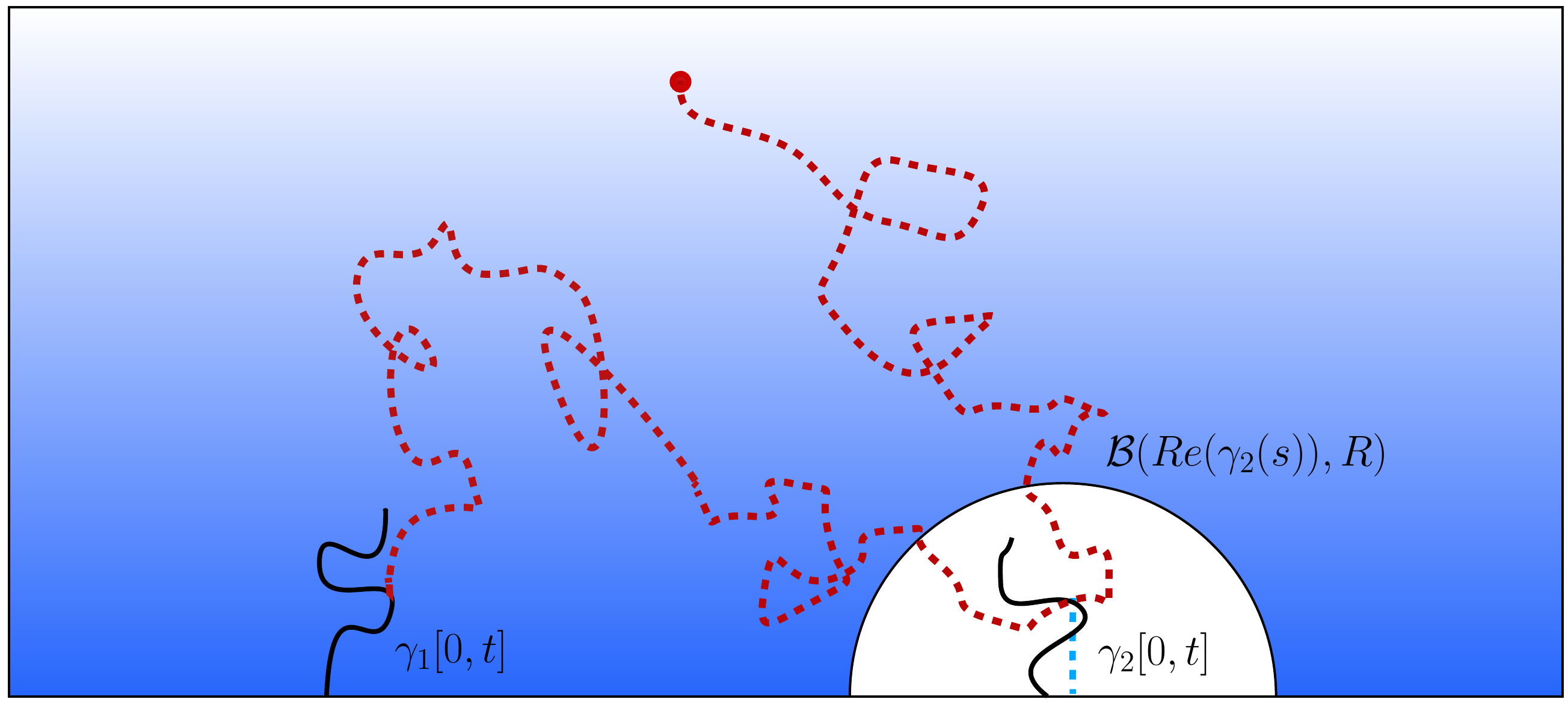}
\caption{A Brownian motion with $\sigma_t<\varrho_t$.}
 \end{figure}

{The following lemma gives a dynamical interpretation of the weights $\lambda$ and $1-\lambda$.}

\begin{lemma}\label{nointer} Let $x(t)=\hcap(\gamma_1(0,t])$ and
$y(t)=\hcap(\gamma_2(0,t])$.
Then $x(t)$ and $y(t)$ are differentiable in $t=0$ with $$\dot{x}(0)=2\lambda \quad \text{and} \quad\dot{y}(0)=2(1-\lambda).$$
\end{lemma}
\begin{proof} Let $U_1,U_2$ be the driving functions for the
Loewner parametrization  ($\gamma_1,\gamma_2)$. Without loss of generality we assume that $\Gamma_1$ is the left slit, i.e. $U_1(t)<U_2(t)$ for all $t\in[0,1].$

\medskip

(i) In a first step we prove
 $x(t)\geq 2\lambda t$ for all $t\in[0,\tau]$ with some $\tau>0$. Let $n\in\N$ and consider the Loewner equation
$$\dot{g}_{t,n}(z)=
\frac{2\alpha_{n,\lambda}(t)}{g_{t,n}(z)-U_{1}(t)}+\frac{2(1-\alpha_{n,\lambda}(t))}{g_{t,n}(z)-U_{2}(t)},
\quad g_{0,n}(z)=z,$$ where $\alpha_{n,\lambda}$ is defined as in
(\ref{Pete}). Let $K_{n,t}$, $t\in[0,1]$, be the corresponding family of
hulls. From Theorem \ref{thm:control}, we know that $\Ha\setminus K_{n,t}\to \Ha\setminus (\gamma_1(0,t]\cup \gamma_2(0,t])$ for $n\to\infty$ in the sense of  kernel convergence.
Let $z_0\in (U_1(0), U_2(0))$ and denote by $z_n(t)$ the solution to

$$\dot{z}_n(t)= \frac{2\alpha_{n,\lambda}(t)}{z_n(t)-U_{1}(t)}+\frac{2(1-\alpha_{n,\lambda}(t))}{z_n(t)-U_{2}(t)}, \quad z_n(0)=z_0.$$ 
It may not exist until $t=1,$ but during its interval of existence we have $z_n(t)-U_{2}(t)<0<z_n(t)-U_{1}(t)$ and 
$$ \frac{2}{z_n(t)-U_2(t)}\leq \dot{z}_n(t) \leq \frac{2}{z_n(t)-U_1(t)}. $$
From this it follows that there exist $\tau,A,B>0$, independent of $n,$ such that $z_n(t)$ exists until $t=\tau$  and $$\max_{s\in[0,\tau]}U_1(s) < A < z_n(t) < B < \min_{s\in[0,\tau]}U_2(s).$$
Thus, for all $n\in \N$ and $t\in(0,\tau]$, we can write $K_{t,n}=C_{t,n}\cup D_{t,n}$, where $C_{t,n}$ and $D_{t,n}$ are disjoint subhulls of $K_{t,n}$ with $$\Ha\setminus C_{t,n} \to \Ha \setminus \gamma_1(0,t], \quad  \Ha\setminus D_{t,n} \to \Ha \setminus \gamma_2(0,t].$$
The cluster sets of $C_{t,n}$ and $D_{t,n}$ with respect to $g_{t,n}$ are sets
$I_1$ and $I_2$ respectively with $I_1\subset (-\infty, z_n(t))$ and
$I_2\subset (z_n(t),+\infty).$ Hence, $C_{t,n}$ is the hull that is growing if
and only if $\alpha_{n,\lambda}(t)=1.$

\medskip
 Let $x_n(t)=\hcap(C_{n,t})$. Then we get 
\begin{eqnarray*}\label{Alfons}x_n\left(\frac{k}{2^n}\right) = &&\sum_{j=1}^k \left(x_n\left(\frac{j}{2^n}\right)-x_n\left(\frac{j-1}{2^n}\right)\right) = \sum_{j=1}^k \left(x_n\left(\frac{j-1+\lambda}{2^n}\right)-x_n\left(\frac{j-1}{2^n}\right)\right) \\ \underset{\text{Lemma \ref{hcap} (c)}}{\geq} &&\sum_{j=1}^k 2\left(\frac{j-1+\lambda}{2^n}-\frac{j-1}{2^n}\right)  =\sum_{j=1}^k \frac{2\lambda}{2^n} = 2\lambda \cdot \frac{k}{2^n} \end{eqnarray*}
for all $n\in \N$ and $k\in\{1,...,2^n\}$ with $k/2^n\leq \tau$.\\
As $x_n(t)\to x(t)$ for every $t\in[0,\tau],$ we conclude that $x(t)\geq 2\lambda t$ for any $t$ of the form $t=k/2^n\leq \tau.$ The set of all those $t$ is dense in $[0,\tau]$ and as $x(t)$ is a continuous function, we deduce 
$x(t)\geq 2\lambda t$ for every $t\in [0,\tau]$.

\medskip

(ii) In a similar way as in step (i), now 
 utilizing the Loewner equation
$$\dot{h}_{t,n}(z)=
\frac{2(1-\alpha_{n,1-\lambda}(t))}{h_{t,n}(z)-U_{1}(t)}+\frac{2\alpha_{n,1-\lambda}(t)}{h_{t,n}(z)-U_{2}(t)}\,
,
\quad h_{0,n}(z)=z\, ,$$ 
we obtain $y(t)\geq 2(1-\lambda)t$ for all $t\ge 0$ small enough.

\medskip

(iii) Using the estimates in (i) and (ii), Lemma  \ref{asyhcap} gives
$$2\lambda t\leq x(t)\leq 2\lambda t+\Landauo(t) \qquad \text{for} \quad t\to0,$$
i.e., $\dot{x}(0)$ exists and $\dot{x}(0)=2\lambda$. In the same way we obtain
$\dot{y}(0)=2(1-\lambda)$.
\end{proof}


\begin{lemma}\label{LipDiff}  
 Let $x(t)=\hcap(\gamma_1(0,t])$.
Then the function $x:[0,1]\to [0,\infty)$ is continuously differentiable with
$$\dot{x}(0)= 2\lambda \quad\text{and} \quad\dot{x}(t)> 2\lambda \quad \text{for all}
\;\; t \in (0,1]\, .$$
In addition,$$\dot{x}(t)=\frac{2 \lambda}{C(x(t),t)}\, , $$
with a continuous function $C : \{(x_0,t) \, : \, 0 \le x_0 \le t, 0 \le t \le
1\} \to (0,1]$, which is continuously differentiable w.r.t.~$t$.
\end{lemma}

\begin{proof}
For $j=1,2$ denote by $\theta_j(s)$  the parameterization  of $\Gamma_j$
 by its half--plane capacity. Let $t\in[0,1]$ and  let $0\le x_0 \leq t.$ Then
 there exists a unique $y_0 \in [0,1]$ such that
 $\theta_1[0,x_0]\cup\theta_2[0,y_0]$ has half--plane capacity $2t$. Apply the
 mapping $A:=g_{\theta_1[0,x_0]}$ on the two slits. We define $\chi(\Delta):=A(\theta_1(x_0+\Delta))$ for all $\Delta \ge 0$ small
 enough. Then we have by Lemma \ref{hcap} (b),
 $\hcap(\chi[0,\Delta])=\hcap(\theta_1[0,x_0+\Delta])-\hcap(\theta_1[0,x_0])=2\Delta$.
 Next we apply the mapping $B:=g_{A(\theta_2[0,y_0])}.$ Let $\psi(\Delta):=B(\chi(\Delta))$.  Now we have $$\frac{\hcap(\psi[0,\Delta])}{2\Delta}\to B'(\chi(0))^2 \quad \text{for} \quad \Delta\to 0,$$ see \cite{LSW:2001}, Lemma 2.8. Note that $B'(\chi(0))^2$ depends on $x_0$ and $t$ only. So let us define the function $C(x_0,t):=B'(\chi(0))^2.$ $C$ has the following properties:
 \begin{itemize} \item $(x_0,t)\mapsto C(x_0,t)$ is continuous: $A$ and $\chi(0)$ depend continuously on $x_0$. Furthermore, $y_0$ depends continuously on the pair $(x_0,t),$ so $B$ depends continuously on $(x_0,t)$ as well as $C(x_0,t)=B'(\chi(0))^2.$  
\item $C$ is continuously differentiable with respect to $t$: For fixed $x_0$, both the value $y_0$ and the mapping $B$ are continuously differentiable with respect to $t,$ see section 4.6.1 in \cite{Lawler:2005}. Hence, as $\chi(0)$ is fixed, also $B'(\chi(0))$ is continuously differentiable w.r.t.~$t$.
\item $C(x_0,t)\in (0,1)$ for all $0\leq x_0 < t \leq 1$: see Proposition 5.15 in \cite{Lawler:2005}.
\end{itemize}
Now we look at the case $x_0=x(t).$ Then
$x(t+h)-x(t)=\hcap(\chi[0,\Delta(h)])=2\Delta(h)$ and we know
that $$\lim_{h\downarrow 0}\frac{\hcap(\psi[0,\Delta(h)])}{h}=2\lambda.$$
This follows by applying Lemma \ref{nointer} to the slit $g_t(\Gamma_1\setminus \gamma_1[0,t])$. Thus
 $$\lim_{h\downarrow 0}\frac{x(t+h)-x(t)}{h}=\lim_{h\downarrow
  0}\frac{2\Delta(h)}{h}=\lim_{h\downarrow 0}\frac{2\Delta(h)\cdot
  \hcap(\psi[0,\Delta(h)])}{\hcap(\psi[0,\Delta(h)])\cdot
  h}=\frac{2\lambda}{C(x(t),t)}.$$  Hence, the right derivative of $x(t)$ exists
and is continuous, so $x(t)$ is continuously differentiable, see Lemma 4.3 in
\cite{Lawler:2005}, and 
$$\dot{x}(t)=\frac{2\lambda}{C(x(t),t)}\, .$$
\end{proof}

\section{\label{uniqueness} Proof of Theorem \ref{Charlie2}, Part II (Uniqueness)}

Let $\nu, \mu\in[0,1]$ be constant weights  and  $U_1,U_2,V_1,V_2:[0,1]\to\R$
be continuous driving functions such that the solutions $g_t$ and $h_t$ of
 $$\dot{g}_t=\frac{2\nu}{g_t-U_1(t)}+\frac{2(1-\nu)}{g_t-U_2(t)},\, g_0(z)=z \quad\text{and} \quad\dot{h}_t=\frac{2\mu}{h_t-V_1(t)}+\frac{2(1-\mu)}{h_t-V_2(t)}, \,h_0(z)=z$$  satisfy $g_1=h_1=g_{\Gamma_1\cup \Gamma_2}.$

\medskip

Assume $\nu >\mu$.
Let $x_1(t)$ and $x_2(t)$ be the half--plane capacities of the generated part
of $\Gamma_1$ at time $t$ with respect to $g_t$ and $h_t$, and let $y_1(t)$
and $y_2(t)$ be the corresponding half--plane capacities of $\Gamma_2.$ Then
$\dot{x}_1(0)=\nu>\mu=\dot{x_2}(0)$ by Lemma \ref{nointer}. Consequently,
$x_1(t)>x_2(t)$ for all $t \ge 0$ small enough. 
Since $x_1(1)=\hcap(\Gamma_1)=x_2(1)$, there is a first time $\tau \in (0,1]$
such that $x_1(\tau)=x_2(\tau)$. Then $x_1(t)>x_2(t)$ for every $t \in
(0,\tau)$, so $\dot{x}_1(\tau)\le \dot{x}_2(\tau)$. On the other hand,
 Lemma \ref{LipDiff} shows that
$$ \dot{x}_1(\tau)=\frac{2 \nu}{C(x_1(\tau),\tau)} > \frac{2
  \mu}{C(x_2(\tau),\tau)}=\dot{x}_2(\tau) \, , $$
a contradiction. Hence we know that $\nu \le \mu$. By switching the roles of
$\nu$ and $\mu$, we deduce $\nu=\mu$.

\medskip

Next, again with the help of  Lemma \ref{LipDiff}, we see that both
functions $x_1$ and $x_2$ are solutions to the same initial value problem,
$$ \dot{x}(t)=\frac{2 \mu}{C(x(t),t)}\, , \quad x(0)=0\, , $$
where $(x_0,t)\mapsto 2\mu/C(x_0,t)$ is continuous, positive and Lipschitz
continuous in $t$. However, the solution to such a problem is
unique according to Theorem 2.7 in \cite{Cid:2003}. Hence $x_1=x_2$ and also
$y_1=y_2$, so
we have $$H(t):=\theta_1[0,x_1(t)]\cup \theta_2[0,y_1(t)] =
\theta_1[0,x_2(t)]\cup \gamma_2[0,y_2(t)]$$ for all $t.$ Using the geometric
meaning of the driving functions (see Section \ref{sec:two--slit}), we finally get
$ U_j(t)=g_{H(t)}(\theta_j(x_1(t))=g_{H(t)}(\theta_j(x_2(t))=V_j(t)$ for $j=1,2$.
This completes the proof of the uniqueness statement of Theorem \ref{Charlie2}.

\bibliographystyle{amsalpha}

\begin{thebibliography}{BBK05}

\bibitem[BBK05]{MR2187598}
M.~Bauer, D.~Bernard, and K.~Kyt{\"o}l{\"a}, \emph{Multiple
  {S}chramm-{L}oewner evolutions and statistical mechanics martingales}, J.
  Stat. Phys. \textbf{120} (2005), no.~5-6, 1125--1163.

\bibitem[BCD12]{BCD:2012}
F.~Bracci, M.~D.~Contreras, S.~D\'{i}az-Madrigal, \emph{Evolution families and the
Loewner equation. I: The unit disc},  J.~Reine Angew.~Math. \textbf{672}
(2012), 1--37.


\bibitem[BS]{BS} Ch.~B\"ohm and S.~Schlei{\ss}inger, preprint.

\bibitem[Car03]{1038.82074}
J.~Cardy, \emph{Stochastic {L}oewner evolution and {D}yson's circular
  ensembles}, J. Phys. A \textbf{36} (2003), no.~24, L379--L386.

\bibitem[CM02]{MR1945279}
L.~Carleson and N.~Makarov, \emph{Laplacian path models}, J. Anal. Math.
  \textbf{87} (2002), 103--150, Dedicated to the memory of Thomas H. Wolff.

\bibitem[CP03]{Cid:2003}
J.~Angel Cid and Rodrigo~L\'{o}pez Pouso, \emph{On first-order ordinary
  differential equations with nonnegative right-hand sides}, Nonlinear Anal.
  \textbf{52} (2003), no.~8, 1961--1977.

\bibitem[Dub07]{MR2358649}
J.~Dub{\'e}dat, \emph{Commutation relations for {S}chramm-{L}oewner
  evolutions}, Comm. Pure Appl. Math. \textbf{60} (2007), no.~12, 1792--1847.

\bibitem[DV11]{PhysRevE.84.051602}
M.~A. Dur\'an and  G.--L. Vasconcelos, \emph{Fingering in a channel and
  tripolar {L}oewner evolutions}, Phys. Rev. E \textbf{84} (2011), 051602.


\bibitem[FR75]{FR75} W.~H.~Fleming and R.~W.~Rishel, 
\emph{Deterministic and Stochastic Optimal Control}, Springer, 1975.

\bibitem[GB92]{GB92} V.~V.~Goryainov and I.~Ba, \emph{Semigroup of conformal
    mappings of the upper half--plane into itself with hydrodynamic
    normalization at infinity}, Ukrainian Math.~J. \textbf{44} (1992), no.~10, 1209--1217.

\bibitem[Gra07]{Graham2007}
K.~Graham, \emph{On multiple {S}chramm--{L}oewner evolutions}, Journal of
  Statistical Mechanics: Theory and Experiment \textbf{2007} (2007), no.~03.

\bibitem[GS08]{MR2495460}
T.~Gubiec and P.~Szymczak, \emph{Fingered growth in channel geometry: a
  {L}oewner-equation approach}, Phys. Rev. E (3) \textbf{77} (2008), no.~4,
  041602, 12.

\bibitem[GM13]{GM} P.~Gumenyuk and A.~del Monaco, \emph{Chordal Loewner equation},
  eprint arxiv:1302.0898v2

\bibitem[Jur97]{Jur97} V.~Jurdjevic, \emph{Geometric Control Theory},
  Cambrdige Univ.~Press, 1997.


\bibitem[KL07]{MR2310306}
M.~J. Kozdron and G.~F. Lawler, \emph{The configurational measure on
  mutually avoiding {SLE} paths}, Universality and renormalization, Fields
  Inst. Commun., vol.~50, Amer. Math. Soc., Providence, RI, 2007, pp.~199--224.

\bibitem[KSS68]{MR0257336}
P.~P. Kufarev, V.~V. Sobolev, and L.~V. Spory{\v{s}}eva, \emph{A certain method
  of investigation of extremal problems for functions that are univalent in the
  half--plane}, Trudy Tomsk. Gos. Univ. Ser. Meh.-Mat. \textbf{200} (1968),
  142--164.



\bibitem[LLN09]{MR2576752}
S.~Lalley, G.~F. Lawler, and H.~Narayanan, \emph{Geometric interpretation of
  half--plane capacity}, Electron. Commun. Probab. \textbf{14} (2009), 566--571.

\bibitem[Law05]{Lawler:2005}
G.~F. Lawler, \emph{Conformally invariant processes in the plane}, Mathematical
  Surveys and Monographs, vol. 114, American Mathematical Society, Providence,
  RI, 2005.

\bibitem[LSW01]{LSW:2001}
G.~F.~Lawler, O.~Schramm, and W.~Werner,
\emph{Values of Brownian intersection exponents. I: Half--plane exponents},
Acta Math.~\textbf{187}, No.~2 (2001), 237--273.


\bibitem[LMR]{LMR10} J.~Lind, D.E.~Marshall, and S.~Rohde, \emph{Collisions
    and spirals of Loewner traces}, Duke Math.~J. \textbf{154} (2010), no.~3, 527--573.

\bibitem[LR]{LR} J.~Lind and S.~Rohde, \emph{Spacefilling curves and phases of
    of the Loewner equations}, Indiana Univ.~Math.~J., to appear, eprint arXiv:1103.0071.

\bibitem[Loe23]{Loewner:1923} K.~L\"owner, \emph{Untersuchungen \"uber
    schlichte konforme Abbildungen des Einheitskreises. I.},
  Math.~Ann. \textbf{89} (1923), 103--121.

\bibitem[Pes36]{Peschl}
E.~Peschl, \emph{Zur Theorie der schlichten Funktionen}, J. Reine Angew. Math.
  \textbf{176} (1936), 61--94.

\bibitem[Pro93]{Prokhorov:1993}
D.~V.~Prokhorov, \emph{Reachable set methods in extremal problems for univalent
  functions}, Saratov University, 1993.

\bibitem[RW]{RW} S.~Rohde and C.~Wong, \emph{Half--plane capacity and
    conformal radius}, Proc. Amer.~Math.~Soc., to appear, eprint arXiv:1201.5878.


\bibitem[RR94]{RR94} M.~Rosenblum and J.~Rovnyak, \emph{Topics in Hardy Classes
    and Univalent Functions}, Birkh\"auser 1994.


\bibitem[Sch00]{Schramm:2000} O.~Schramm, \emph{Scaling limits of loop-erased random
  walks and uniform spanning trees},  Isr.~J. Math. \textbf{118} (2000), 221--288.

\end{thebibliography}

\vspace{1.0cm}
 {\footnotesize
 \textsc{University of W\"urzburg, Department of Mathematics, 97074 W\"urzburg, Germany}\\
 \textit{E-mail addresses:}\\{ \tt roth@mathematik.uni-wuerzburg.de,\\
   sebastian.schleissinger@mathematik.uni-wuerzburg.de}}

\end{document}